\documentclass{l4dc2026}

\title[]{Online Tracking with Predictions for Nonlinear
Systems with Koopman Linear Embedding}
\usepackage{times}

\author{\Name{Chih-Fan Pai} \Email{cpai@ucsd.edu}\\
 \Name{Xu Shang} \Email{x3shang@ucsd.edu}\\
 \Name{Jiachen Qian} \Email{jiq012@ucsd.edu}\\
 \Name{Yang Zheng} \Email{zhengy@ucsd.edu}\\
 \addr Department of Electrical and Computer Engineering, University of California San Diego}

\newtheorem{thm}{Theorem}[section]
\newtheorem{lem}{Lemma}[section]
\newtheorem{prop}{Proposition}[section]
\newtheorem{cor}{Corollary}[section]
\newtheorem{exmp}{Example}[section]
\newtheorem{defn}{Definition}[section]

\newtheorem{assum}{Assumption}[section]

\usepackage{enumitem}
\usepackage{float}
\usepackage{wrapfig}

\usepackage{times}
\usepackage{optidef}
\usepackage{verbatim}
\usepackage{commath}

\newcommand{\tr}{\mathsf{ T}}

\DeclareMathOperator*{\argmin}{\arg\!\min}

\usepackage{comment}
\usepackage{microtype}
\usepackage{graphicx}
\usepackage{subcaption,setspace}
\usepackage{booktabs} 

\usepackage{hyperref}

\usepackage[capitalize,noabbrev]{cleveref}

\usepackage{caption}

\newcommand{\ini}{\textnormal{ini}}
\newcommand{\f}{\textnormal{F}}

\newcommand{\D}{\textnormal{d}}

\usepackage{algorithm}
\usepackage{algorithmic}

\makeatletter
\g@addto@macro\normalsize{%
  \setlength\abovedisplayskip{5pt}%
  \setlength\belowdisplayskip{5pt}%
  \setlength\abovedisplayshortskip{5pt}%
  \setlength\belowdisplayshortskip{5pt}%
}
\makeatother

\begin{document}

\maketitle

\vspace{-5mm}

\begin{abstract}%
We study the problem of online tracking in unknown nonlinear dynamical systems, where only short-horizon predictions of future target states are available. This setting arises in practical scenarios where full future information and exact system dynamics are unavailable. We focus on a class of nonlinear systems that admit a Koopman linear embedding, enabling the dynamics to evolve linearly in a lifted space. Exploiting this structure, we analyze a model-free predictive tracking algorithm based on Willems' fundamental lemma, which imposes dynamic constraints using only past data within a receding-horizon control framework. We show that, for Koopman-linearizable systems, the cumulative cost and dynamic regret of the nonlinear tracking problem coincide with those of the lifted linear counterpart. Moreover, we prove that the dynamic regret of our algorithm decays exponentially with the prediction horizon, as validated by numerical experiments.
\end{abstract}


\section{Introduction}

We study the online tracking problem in unknown nonlinear dynamical systems, where an agent aims to follow a time-varying target trajectory using only short-horizon predictions of future target states. This sequential decision-making problem arises in many real-world applications, such as robotics \citep{falcone2007predictive}, autonomous systems \citep{rios2016survey,zheng2016distributed}, and adaptive control \citep{krstic1995nonlinear,ioannou1996robust,annaswamy2021historical}, where complete knowledge of future targets or system dynamics is rarely available. At each time step, the agent observes the current system state, receives short-horizon predictions of future target states, and selects a control action to steer the system toward the evolving target. The goal is to minimize the cumulative tracking error as well as its control efforts over a finite horizon.

While online control or tracking of linear dynamical systems has received significant attention \citep{abbasi2014tracking,li2019online,yu2020power,karapetyan2023online,tsiamis2024predictive,
agarwal2019online,foster2020logarithmic,zhang2022adversarial,hazan2022introduction}, its extension to nonlinear systems, especially in the presence of predictions, remains largely unexplored. In this work, we focus on a special class of nonlinear systems that admit a \emph{Koopman linear embedding}: a lifting procedure into a higher-dimensional space in which the system evolution becomes linear \citep{brunton2016koopman,shang2024willems}. This structure enables us to reformulate the original nonlinear problem as a linear one in the lifted space, allowing us to exploit tools from linear control. 
However, in practice, the exact Koopman embedding is often unknown. To address this, we study a model-free predictive tracking algorithm inspired by Willems' fundamental lemma \citep{williams2015data}. While the method follows the general template of Model Predictive Control (MPC) \citep{garcia1989model}, a widely used approach for online control with predictions, it departs from classical MPC by imposing nonlinear dynamics constraints purely from offline trajectory data \citep{shang2024willems}, without requiring explicit system identification.

The asymptotic properties of MPC have been extensively studied under general assumptions on dynamics and cost functions \citep{diehl2010lyapunov,grune2020economic,6099558}. However, non-asymptotic performance metrics remain less well understood. \textit{Dynamic regret}, a standard measure in online learning \citep{jadbabaie2015online} and control \citep{yu2020power,goel2020regret,zhang2021regret,lin2022bounded}, quantifies the performance gap between an online controller and the optimal control sequence in hindsight. This metric is particularly suited for nonstationary environments, such as tracking moving or adversarial targets \citep{abbasi2014tracking,dressel2019hunting}. Recent work has established dynamic regret guarantees for linear MPC, demonstrating near-optimal performance in both stochastic and adversarial disturbance settings \citep{yu2020power,zhang2021regret,lin2021perturbation}. However, these results rely on known linear dynamics, limiting their applicability to nonlinear or model-free settings. Extending dynamic regret guarantees to such systems remains a significant challenge due to the complexity of nonlinear dynamics and the absence of explicit models.
To our knowledge, this work is the first to provide \textit{dynamic regret} guarantees for online tracking in unknown Koopman-linearizable nonlinear systems.

\textbf{Our contributions.} More specifically, for systems that admit exact Koopman embeddings, we establish an equivalence between the nonlinear tracking problem and its linear counterpart in the Koopman space.
    In the offline setting with full target information, we show that the optimal tracking policy can be characterized by solving the lifted linear problem (Lemma \ref{lem:cost-eq}).
    Moreover, for any online controller, its dynamic regret can be analyzed by the lifted linear counterpart (Corollary \ref{coro:reg-eq}).
    Building on this insight, we study a model-free predictive tracking method (Algorithm \ref{alg:MPC}) based on the extended Willems's fundamental lemma \citep{shang2024willems}, which enforces the nonlinear dynamics constraint using offline trajectory data and computes control actions by solving a data-driven quadratic program at each time step. 
    In particular, we establish the first dynamic regret bound, showing that the regret grows linearly with full horizon and decays exponentially with the prediction horizon (Theorem \ref{thm:regret-MPC}).
    Our design and analysis depart from standard approaches in two key aspects.
    First, stability (state boundedness) is ensured by a sufficiently long prediction horizon rather than a carefully designed terminal cost. Second, unlike most regret analyses that rely critically on positive definite stage costs, a condition violated by the Koopman structure, our results hold under the weaker assumptions of positive semidefinite stage costs and detectability.

\textbf{Related works.} Our work intersects with two key areas: Koopman operator theory and online~(nonstochastic) control. We here review some closely related results.

\vspace{-5mm}

\begin{itemize}
    \item[(i)] \textbf{Koopman operator theory}, which enables representing nonlinear dynamical systems as linear systems in a high-dimensional space via lifting functions \citep{brunton2016koopman}. Standard approaches often approximate the infinite-dimensional Koopman operator using Extended Dynamic Mode Decomposition \citep{williams2015data,proctor2018generalizing}, which requires a careful selection of lifting functions to avoid modeling bias \citep{korda2018linear,berberich2020trajectory}. Recent advances learn embeddings \citep{shi2022deep} or bypass explicit lifting \citep{shang2024willems}, but these methods typically assume known or partially identified dynamics. In contrast, our work assumes exact Koopman embeddings and employs a model-free predictive control algorithm inspired by a recent Willems’ fundamental lemma \citep{shang2024willems}. By imposing dynamics constraints through offline trajectory data, we eliminate the need for system identification or lifting approximation, enabling robust tracking of unknown nonlinear systems with theoretical guarantees (Theorem \ref{thm:regret-MPC}). 
    \vspace{-2mm}
    \item[(ii)] \textbf{Online nonstochastic control,} which adopts a learning-theoretic perspective under nonstochastic or adversarial environments. This class of approaches typically reformulates the decision-making problem into online convex optimization, intending to minimize \emph{policy regret}, which compares the performance of the online policy against the best fixed policy over a parameterized class \citep{agarwal2019online,simchowitz2020improper,hazan2022introduction}. While this framework allows for tractable algorithm design, it is less suited for nonstationary tasks such as target tracking with evolving references. Instead, we focus on \emph{dynamic regret}, which measures the performance gap relative to the best sequence of control actions in hindsight. This stronger benchmark better captures the challenges of online tracking under predictions. Moreover, while most existing results apply to linear systems with known dynamics, our work provides dynamic regret guarantees in a nonlinear, model-free setup.
\end{itemize}

\vspace{-2mm}
The rest of the paper is organized as follows. Section \ref{sec:problem-setup} presents the online tracking problem for nonlinear systems. Section \ref{sec:nonlinear-to-linear} establishes the equivalence between the original nonlinear problem and its lifted linear counterpart. Section \ref{sec:predictive-tracking-alg} presents a unified predictive tracking algorithm. Section \ref{sec:regret-analysis} analyzes the dynamic regret. Section \ref{sec:experiment} provides numerical experiments, and Section \ref{sec:conclusion} concludes the paper. Technical proofs and experimental details are given in the appendix. 

\textbf{Notation.}
For a matrix $M$, let $\|M\|$ denote its induced $\ell_2$-norm, and let $\rho(M)$ denote its spectral radius, defined as the maximum absolute value of its eigenvalues. For a vector $x$, $\|x\|$ denotes the standard Euclidean norm, and $\|x\|_P := (x^\tr P x)^{1/2}$ denotes the $P$-weighted norm for any $P \succeq 0$. We write $x_{1:t} = \{x_1, \dotsc, x_t\}$ for a sequence of vectors and also use it to denote $[x_1^\tr, \ldots, x_t^\tr]^\tr$ for notation simplicity. We utilize $[t] = \{1, \dotsc, t\}$ for the index set.

\vspace{-2mm}
\section{Problem Setup} \label{sec:problem-setup}

\vspace{-1mm}

We study an online control problem where the system initialized at $z_1$ evolves according to some unknown nonlinear dynamics $z_{t+1} = f(z_t, u_t)$. Here $z_t \in \mathbb{R}^{n_z}$ is the state and $u_t \in \mathbb{R}^{n_u}$ is the control action at time $t$. The agent aims to track a target trajectory $r_{1:T}$ over a finite horizon of length $T$. 
In the classical setting, the entire trajectory is known a priori \citep{anderson2007optimal,bertsekas2012dynamic}. However, in many applications the target may be unknown or even adversarial.
We therefore consider the problem of \textit{online tracking with predictions}, where at each time step $t\in[T]$, the agent interacts with the adversary (environment) as follows:
\vspace{-2mm}
\begin{itemize} \setlength{\itemsep}{1pt}
    \item the agent observes $z_t$ and predictions $r_{t:t+W-1}$ over a prediction horizon $W$;
    \vspace{-2mm}
    \item the agent selects and applies a control action $u_t$;
    \vspace{-2mm}
    \item the adversary selects $r_{t+W}$, and the agent incurs the stage cost $\|z_t - r_t\|_{Q_z}^2 + \|u_t\|_{R}^2$;
    \vspace{-2mm}
    \item the system state evolves according to $z_{t+1} = f(z_t, u_t)$.
\end{itemize}

\vspace{-1mm}
Let $\pi \!=\! \{\pi_t\}_{t=1}^{T}$ denote the online (causal) control policy followed by the agent. Based on the above interaction protocol, $\pi_t$ can depend only on the available information up to that point, i.e., $u_t \!=\! \pi_t(z_{1:t}, r_{1:t+W\!-\!1})$. 
Let $\mathbf{u}=u_{1:T}$ and $\mathbf{r}=r_{1:T}$. We define the cumulative cost incurred by $\pi$ as
\vspace{-2mm}
\begin{equation}
    J_T(\mathbf{u};\mathbf{r}) = \sum_{t=1}^T\ell(z_t,u_t;r_t), \quad \ell(z_t,u_t;r_t):=\|z_t - r_t\|_{Q_z}^2 + \|u_t\|_R^2,
    \vspace{-1mm}
\end{equation}
where $Q_z \! \succeq \!0$ and $R \! \succ \! 0$. 
For notation simplicity, we omit the dependence of $J_T$ on $z_1$.

\noindent \textbf{Dynamic regret.} We measure the performance of the agent's policy $\pi$ via the notion of dynamic regret, which compares the cumulative cost of $\pi$ to that of the optimal noncausal policy $\pi^*$ in hindsight (with full knowledge of the target trajectory $\mathbf{r}$ and system dynamics $f$):
\begin{equation} \label{eq:dynamic-reg}
    \mathrm{Reg}_T(\pi) = J_T(\mathbf{u};\mathbf{r}) - J_T^*,
\end{equation}
where $J_T^* := \min_{\mathbf{u}} \left\{ J_T(\mathbf{u}; \mathbf{r})  \mid  z_{t+1} = f(z_t, u_t) \right\}$.  Dynamic regret reflects how well the online controller adapts to a nonstationary, possibly adversarial environment \citep{goel2020regret}. This contrasts with policy regret \citep{agarwal2019online}, which compares against the best fixed policy from a parameterized class. We make no assumptions on the target trajectory $\mathbf{r}$ beyond boundedness; without loss of generality, we assume $\|r_t\|\leq 1$ for all $t\in[T]$.
In this paper, we focus on a special class of nonlinear systems \citep{shang2024willems, brunton2016koopman}.

\vspace{1mm}
\noindent \textbf{Koopman-linearizable dynamics.}
A (nonlinear) system $z_{t+1} = f(z_t, u_t)$ is said to be \emph{Koopman-linearizable} if there exist a lifting function $\psi:\mathbb{R}^{n_z} \to \mathbb{R}^{n_x}$ and matrices $A \!\in\! \mathbb{R}^{n_x \times n_x}$, $B \!\in\! \mathbb{R}^{n_x \times n_u}$, and $C \!\in\! \mathbb{R}^{n_z \times n_x}$ such that, for all $z_t \!\in\! \mathbb{R}^{n_z}$, the lifted state $x_t = \psi(z_t)$ evolves linearly $x_{t+1} = A x_t + B u_t$ and the original state satisfies $z_t = C x_t = C\psi(z_t)$.
We refer to the tuple $(A, B, C, \psi)$ as the Koopman embedding.
This class of nonlinear dynamics strictly goes beyond standard linear systems and can be used as an alternative to classical local linearization techniques. 
\vspace{-2mm}
\begin{assum}
\label{assump:koopman}
The nonlinear system $f$ is Koopman-linearizable. In addition, there exists a uniform bound $D_\psi< \infty$ such that $\|\psi(r)\| \leq D_\psi$ for all $\|r\|\leq1$.
\end{assum}

\vspace{-3mm}
\begin{wrapfigure}[5]{r}{.47\textwidth}
\vspace{-8mm}
\begin{equation}
\label{eq:slow-mani}
\begingroup
    \setlength\arraycolsep{3pt}
\def\arraystretch{0.85} 
\begin{bmatrix}
    z_{1}^{+} \\
    z_{2}^{+}
\end{bmatrix}
\!=\!
\begin{bmatrix}
    0.99 z_1 \\
    z_2 + z_1^2
\end{bmatrix}
\!+\!
\begin{bmatrix}
    0 \\
    1
\end{bmatrix} u.
\endgroup
\end{equation}
\begin{equation}
\label{eq:slow-mani-Koop}
\begingroup
    \setlength\arraycolsep{3pt}
\def\arraystretch{0.85} 
    \begin{bmatrix}
    x_1^{+} \\
    x_2^{+} \\
    x_3^{+}
\end{bmatrix}\!=\!
\begin{bmatrix}
    0.99 & 0 & 0 \\
    0 & 1 & 1 \\
    0 & 0 & 0.99^2
\end{bmatrix}
\begin{bmatrix}
    x_1 \\
    x_2 \\
    x_3
\end{bmatrix}
\!+\!
\begin{bmatrix}
    0 \\
    1 \\
    0
\end{bmatrix} u.
\endgroup
\end{equation}
\end{wrapfigure} 
\begin{exmp}[\citep{brunton2016koopman}] \label{exmp:slow-manifold}
Consider the nonlinear system \eqref{eq:slow-mani}.
With the lifted coordinates \( x = \psi(z) := [z_1,\ z_2,\ z_1^2]^\tr \), it admits an exact~linear embedding \eqref{eq:slow-mani-Koop}. \hfill $\square$
\end{exmp}
\vspace{-2mm}

The lifting function $\psi$ is generally unknown even if it exists. In this work, we develop a model-free predictive controller exploiting the Koopman structure without requiring explicit knowledge of $f$ or $\psi$, and provides a dynamic regret guarantee.

\section{From Nonlinear to Linear via Koopman Embedding} \label{sec:nonlinear-to-linear}
\vspace{-1.5mm}
We establish a formal equivalence between the original nonlinear tracking problem and its lifted linear counterpart, then use linear control theory to characterize the optimal offline noncausal policy.

\vspace{-2mm}
\subsection{Equivalence in the lifted space}
\vspace{-1mm}
Consider the nonlinear dynamics $f$ initialized at $z_1$ with Koopman embedding $(A,B,C,\psi)$. Let $x_1=\psi(z_1)$.
Define $Q:=C^\tr Q_z C$ and the lifted cumulative cost
\vspace{-1mm}
\[
    {J}_T^{\mathrm{lft}}(\mathbf{u};\mathbf{r}):= \sum_{t=1}^T \ell_{\mathrm{lft}}(x_t,u_t;r_t), \quad \ell_{\mathrm{lft}}(x_t,u_t;r_t):=\|x_t - \psi(r_t)\|_{Q}^2 + \|u_t\|_R^2.
\]
\vspace{-1mm}
We now present a lemma capturing the cost equivalence under Koopman lifting.
\begin{lem} \label{lem:cost-eq}
    Let $x_1=\psi(z_1)$. The following statements hold:
    \vspace{-2mm}
    \begin{enumerate}
        \item[\textnormal{(i)}] ${J}_T(\mathbf{u};\mathbf{r})={J}_T^{\mathrm{lft}}(\mathbf{u}; \mathbf{r})$ for any control sequence $\mathbf{u}$;
        \item[\textnormal{(ii)}] $J_T^*=({J}_T^{\mathrm{lft}})^*$, where $({J}_T^{\mathrm{lft}})^*:=\min_{\mathbf{u}} \ \{{J}_T^{\mathrm{lft}}(\mathbf{u};\mathbf{r}):x_{t+1}=Ax_t+Bu_t \}$.
    \end{enumerate}
\end{lem}
\vspace{-1mm}

\begin{proof}
    The proof is straightforward. Since $z_t \!=\! C x_t$ and $r_t \!=\! C \psi(r_t)$, the stage cost satisfies $\|z_t - r_t\|_{Q_z}^2 \!=\! \|C x_t - C \psi(r_t)\|_{Q_z}^2 \!=\! \|x_t - \psi(r_t)\|_{Q}^2$, and thus (i) follows.
    For (ii), because $J_T$ and $J_T^{\mathrm{lft}}$ agree for all control sequences $\mathbf{u}$, their optimal values are identical.
\end{proof}

We next define the lifted dynamic regret 
$
    {\mathrm{Reg}}_T^{\mathrm{lft}}(\pi) :={J}_T^{\mathrm{lft}}(\mathbf{u};\mathbf{r})-({J}_T^{\mathrm{lft}})^*.
$
By Lemma \ref{lem:cost-eq}, the following result follows immediately.
\begin{cor}[Equivalence of dynamic regrets] \label{coro:reg-eq}
For any online policy $\pi$, $\mathrm{Reg}_T(\pi)=\mathrm{Reg}_T^{\mathrm{lft}}(\pi)$.
\end{cor}
\vspace{-1mm}
This result allows us to analyze the dynamic regret \eqref{eq:dynamic-reg} in the linear Koopman space, while drawing conclusions for the original nonlinear system. 
We impose the following assumption.
\begin{assum} \label{assump:LQ}
    $(A, B)$ is stabilizable, $(Q,A)$ is detectable, $Q_z\succeq0$, $Q\succeq0$, and $R\succ0$.
\end{assum}

The assumption ${Q} \succeq 0$ is nontrivial. Since the recovery matrix $C$ is not full column rank, the lifted cost matrix ${Q} = C^\tr Q_z C$ cannot be positive definite, even if $Q_z \succ 0$. 

\subsection{Optimal offline policy}
\vspace{-1mm}
We now analyze the optimal offline (noncausal) tracking policy by exploiting an exact Koopman embedding that transforms the nonlinear tracking problem into an equivalent linear-quadratic tracking problem in the lifted space. The two formulations are presented below:
\begin{align}
    J_T^* &:= \min_{\mathbf{u}} \textstyle\sum_{t=1}^T \ell(z_t,u_t;r_t) \qquad \mathrm{s.t.} \ z_{t+1} = f(z_t, u_t), \ z_1 \text{ given}; \tag{N} \label{eq:nonlinear-tracking} \\
    (J_T^{\mathrm{lft}})^* &:= \min_{\mathbf{u}} \textstyle\sum_{t=1}^T \ell_{\mathrm{lft}}(x_t,u_t;r_t) 
       \qquad \mathrm{s.t.} \ x_{t+1} = A x_t + B u_t, \ x_1 = \psi(z_1). \tag{L}  \label{eq:LQT} 
\end{align}
To solve the linear lifted problem \eqref{eq:LQT}, we first introduce the standard Riccati recursion.
\begin{defn}[Riccati recursion] \label{def:Riccati-recursion-LQT}
Let $P_{T} = Q$ and define the backward Riccati recursion:
\begin{align*}
    &P_t = Q + A^\tr P_{t+1} A - A^\tr P_{t+1} B \Sigma_t^{-1} B^\tr P_{t+1} A, \quad K_t= \Sigma_t^{-1} B^\tr P_{t+1} A
\end{align*}
\vspace{-3mm}\\ 
for all $t\in[T-1]$ where we denote $\Sigma_t = R + B^\tr P_{t+1} B$.
\end{defn}
\vspace{-2mm}

We also introduce the following notation for brevity.
Let the closed-loop matrix at time $t\in[T]$ be $A_{\mathrm{cl},t} := A - B {K}_t$, and define the state transition matrix for any $t_1 < t_2 < T$ as $A_{\mathrm{cl},t_1 \to t_2} := A_{\mathrm{cl},t_2} A_{\mathrm{cl},t_2-1} \cdots  A_{\mathrm{cl},t_1+1}$ with the convention $A_{\mathrm{cl},t_1 \to t_1} := I$.

Now we characterize the optimal policy for \eqref{eq:LQT}, following \citep{foster2020logarithmic,zhang2021regret,goel2022power}.
A self-contained derivation is given in Appendix~\ref{app:proof-opt-offline-policy}.

\vspace{-2mm}
\begin{thm} \label{thm:opt-offline-tracking}
    Let $\pi_t^*(x_t;\mathbf r)$ denote the optimal offline policy for \eqref{eq:LQT} with optimal control action $u_t^*$, where $x_t\!=\!\psi(z_t)$. 
    The policy admits the following closed-form expression for all $t\in[T]$:
    \vspace{-1mm}
    \begin{equation} \label{eq:opt-LQT}
        u_t^* = \pi_t^*(x_t;\mathbf{r}) = -K_t (x_t - \psi(r_t)) - \textstyle\sum_{i=t}^{T-1} K_{t\to i} \left( A \psi(r_i) - \psi(r_{i+1}) \right),
    \end{equation}
    where $K_t$ denotes the feedback gain and ${K_{t\to i}} \!:=\! \Sigma_t^{-1} B^\tr A_{\mathrm{cl}, t \to i}^\tr P_{i+1}, t\!\leq\! i\!<\!T$, denote the feedforward gains, both computed via the Riccati recursion in Definition~\ref{def:Riccati-recursion-LQT}.
\end{thm}
\vspace{-1mm}
By the cost equivalence established in Lemma \ref{lem:cost-eq}, the following result follows directly.
\begin{cor}
    The optimal control sequence from \eqref{eq:opt-LQT} achieves the minimum cost $J_T^*$ for \eqref{eq:nonlinear-tracking}.
\end{cor}
\vspace{-1mm}

While \eqref{eq:opt-LQT} provides a closed-form solution to \eqref{eq:nonlinear-tracking}, the policy requires full knowledge of the entire future target trajectory $r_{t:T}$ and the exact Koopman embedding, and is therefore not implementable online. Instead, it serves as an offline benchmark for evaluating the dynamic regret of the online controller introduced in the next section.

\vspace{-2mm}
\section{Predictive Tracking for Koopman-linearizable Systems} \label{sec:predictive-tracking-alg}
\vspace{-1mm}

MPC is a widely used online control strategy that repeatedly solves a finite-horizon optimal control problem using short-term predictions. At each step, only the first control action from the optimized sequence is executed, and the procedure repeats.
Below, we introduce a unified predictive tracking algorithm that encompasses three MPC variants, differing in the level of model knowledge:
\vspace{-1.5mm}
\begin{itemize}
\setlength{\itemsep}{0pt}
    \item \textbf{Nonlinear MPC}, which optimizes directly over the true nonlinear dynamics;
    \vspace{-1mm}
    \item \textbf{Lifted Linear MPC}, which assumes access to a known Koopman embedding;
    \vspace{-1mm}
    \item \textbf{Data-driven MPC}, which uses a data-driven dynamics constraint without model knowledge.
\end{itemize}
Note that Algorithm \ref{alg:MPC} follows the standard MPC procedure up to time step $T-W$. After that, a final open-loop control sequence is applied over the remaining horizon.
\begin{algorithm}[t]
\caption{Predictive Tracking}
\label{alg:MPC}
    \begin{algorithmic}[1] 
    \FOR{$t\in[T-W]$}
        \STATE Observe current state $z_t$ and receive target predictions $r_{t:t+W-1}$.
        \STATE Solve \eqref{eq:nonlinear-MPC}, \eqref{eq:linear-MPC}, or \eqref{eq:DDPC}, and apply the  first cntrol from the optimized sequence.
    \ENDFOR
    \STATE At time step $t=T\!-\!W\!+\!1$, observe current state $z_t$ and receive target predictions $r_{t:T}$.
    \STATE Solve \eqref{eq:nonlinear-MPC}, \eqref{eq:linear-MPC}, or \eqref{eq:DDPC}, and apply the optimized sequence over the last $W$ steps.
    \end{algorithmic}
\end{algorithm}
\vspace{-2mm}
\subsection{Predictive control with known dynamics} \vspace{-1mm}
We first consider the case where the nonlinear dynamics $f$ and its Koopman embedding $(A,B,C,\psi)$ are known. In this setting, Algorithm~\ref{alg:MPC} can be instantiated using either the nonlinear formulation \eqref{eq:nonlinear-MPC} or the lifted linear formulation \eqref{eq:linear-MPC}, defined as follows:
\begin{align}
     & \min_{u_{1:W|t}\!=\!(u_{1|t},\hdots,u_{W|t})}
        \textstyle\sum_{i=1}^{W} \ell(z_{i|t},u_{i|t};r_{t+i-1})\quad 
        \mathrm{s.t.}\ z_{i+1|t} \!=\! f(z_{i|t},u_{i|t}),\ z_{1|t}\!=\!z_t; \tag{\small{N-MPC}} \label{eq:nonlinear-MPC} \\
    & \min_{u_{1:W|t}\!=\!(u_{1|t},\hdots,u_{W|t})} \textstyle\sum_{i=1}^{W} \ell_{\mathrm{lft}}(x_{i|t},u_{i|t};r_{t+i-1}) \quad
        \mathrm{s.t.} \ x_{i+1|t} \!=\! A x_{i|t} \!+\! B u_{i|t},\ x_{1|t}\!=\!\psi(z_t) . \tag{\small{L-MPC}} \label{eq:linear-MPC}
\end{align}

This lifted formulation \eqref{eq:linear-MPC} enables the use of linear control techniques. A key observation is that Theorem~\ref{thm:opt-offline-tracking} can be applied to characterize the resulting online MPC policy.

\vspace{-2mm}
\begin{thm} \label{thm:opt-MPC}
Let $\pi_t^{\textnormal{L-MPC}}(x_t; \mathbf{r})$ denote the online (causal) policy at time $t$ obtained by solving \eqref{eq:linear-MPC} in Algorithm~\ref{alg:MPC}, where $x_t\!=\!\psi(z_t)$. The policy admits the following closed-form expression:
\vspace{-5mm}
\begin{itemize}[leftmargin=2em, itemsep=1pt]
    \item[\textnormal{(i)}] For all $t \in [T - W]$, the policy is given by
    \vspace{-1mm}
    \begin{align} \label{eq:MPC-LQT}
        \pi_t^{\textnormal{L-MPC}}(x_t; \mathbf{r})\!=\! -\bar{K}_1 (x_t - \psi(r_t)) 
        \!-\! \textstyle\sum_{i = t}^{t + W - 2} \bar{K}_{1 \to i - t + 1} \left( A \psi(r_i) \!-\! \psi(r_{i+1}) \right),
    \end{align}
    where $\bar{K}_1 \!:=\! K_{T\!-\!W}$ denotes the feedback gain and $\bar{K}_{1 \to k} \!:=\! {K}_{T\!-\!W\to T\!-\!W+k}$, $k\!\in\![W\!-\!1]$, denote the feedforward gains, as defined by the Riccati recursion in Definition~\ref{def:Riccati-recursion-LQT} and Theorem~\ref{thm:opt-offline-tracking}. 

    \item[\textnormal{(ii)}] For $T - W \!<\! t \!\leq\! T$, the policy coincides with the optimal offline policy from Theorem \ref{thm:opt-offline-tracking}, i.e., 
    \[
    \pi_t^{\textnormal{L-MPC}}(x_t; \mathbf{r}) = \pi_t^*(x_t; \mathbf{r}).
    \]
\end{itemize}
\end{thm}

\vspace{-2mm}
For $t\in[T - W]$, the MPC gains $\bar{K}_1$ and $\{\bar{K}_{1 \to k}\}_{k=1}^{W}$ in \eqref{eq:MPC-LQT} are time-invariant, in contrast to the time-varying offline optimal gains $K_t$ and $\{K_{t\to i}\}_{i=t}^{T-1}$ in Theorem~\ref{thm:opt-offline-tracking}.
This invariance stems from the constant weight matrices $Q$ and $R$. If the weights were time-varying, the resulting MPC gains would also vary with time~\citep{zhang2021regret}.
The following result is immediate from Lemma \ref{lem:cost-eq}.
\begin{cor}
    Let $\pi_t^{\textnormal{N-MPC}}(z_t; \mathbf{r})$ denote the online (causal) policy at time $t$ obtained by solving \eqref{eq:nonlinear-MPC} in Algorithm \ref{alg:MPC}. Then $\pi_t^{\textnormal{N-MPC}}(z_t; \mathbf{r})=\pi_t^{\textnormal{L-MPC}}(x_t; \mathbf{r})$ for all $t\in[T]$ where $x_t=\psi(z_t)$.
\end{cor}

\vspace{-3mm}
\subsection{Data-driven predictive control with unknown dynamics}
\vspace{-1mm}
We now consider the case where both the dynamics and Koopman embedding are unknown, but a sufficiently long input-state trajectory $(\mathbf{u}_\D := \bar{u}_{1:n_\D}, \mathbf{z}_\D := \bar{z}_{1:n_\D})$ is available.
Using only this data, we design a predictive controller referred to as
\textit{data-driven predictive control} (DDPC).

\textbf{Online DDPC for tracking.}
We instantiate Algorithm \ref{alg:MPC} by solving:
\begin{equation*} 
     \min_{u_{1:W|t}, z_{1:W|t}, g} \textstyle\sum_{i=1}^{W}  \|z_{i|t} - r_{t+i-1}\|_{Q_z}^2 \!+\! \|u_{i|t}\|_R^2 
     \quad \mathrm{s.t.} \
     H_\D {g} \!=\! [\mathbf{u}_{\ini,t}^\tr, u_{1:W | t}^\tr, \mathbf{z}_{\ini,t}^\tr, z_{1:W | t}^\tr]^\tr, \tag{\small{DDPC}} \label{eq:DDPC}
\end{equation*}
where $H_\D$ is a data matrix constructed from the offline trajectory $(\mathbf{u}_\D, \mathbf{z}_\D)$ as described below. Moreover, $\mathbf{u}_{\ini,t} \!:=\! u_{t-T_\ini:t-1} \in \mathbb{R}^{n_u T_\ini}$ and $\mathbf{z}_{\ini,t} \!:=\! z_{t-T_\ini:t-1} \in \mathbb{R}^{n_z T_\ini}$ denote the most recent length-$T_\ini$ input-state trajectory, which serves as the initial condition for the predictive control.
With a sufficiently rich library $H_\D$ and long $(\mathbf{u}_{\ini,t}, \mathbf{z}_{\ini,t})$, the equality constraint provides a model-free representation of the nonlinear system, as guaranteed by the extended Willem's fundamental lemma \citep{shang2024willems}.  
The formulation \eqref{eq:DDPC} offers two key advantages: (i) it bypasses explicit system identification and the need for choosing lifting functions, and (ii) it converts the nonlinear constraint into a linear one.
As shown below, the resulting control actions coincide with those from \eqref{eq:nonlinear-MPC} and \eqref{eq:linear-MPC}.
We next detail the data-driven constraint and the data requirement. 

\textbf{Construction of data library.}
The matrix $H_\D$ is obtained by partitioning the length-$n_\D$ trajectory $(\mathbf{u}_\D, \mathbf{z}_\D)$ into $l$ overlapping segments of length $L\!:=\! T_\ini + W$, where $l \!=\! n_\D -L + 1$.
Specifically,
\vspace{-3mm}
\begin{equation*}
\begin{small}
\begingroup
    \setlength\arraycolsep{2pt}
\def\arraystretch{0.9} 
H_\D :=
\begin{bmatrix}
    \mathbf{u}_\D^1 & \mathbf{u}_\D^2 & \cdots & \mathbf{u}_\D^l \\
    \mathbf{z}_\D^1 & \mathbf{z}_\D^2 & \cdots & \mathbf{z}_\D^l
\end{bmatrix}
= \begin{bmatrix}
    \mathcal{H}_L(\mathbf{u}_\D)  \\ \mathcal{H}_L(\mathbf{z}_\D)
\end{bmatrix}, \quad \mathcal{H}_L(w_{1:n_\D}) := \begin{bmatrix}
    w_1 & w_2 & \cdots & w_{n_\D-L+1} \\
    w_2 & w_3 & \cdots & w_{n_\D-L+2} \\
    \vdots & \vdots & \ddots & \vdots \\
    w_L & w_{L+1} & \cdots & w_{n_\D}
\end{bmatrix}, 
\endgroup
\end{small}
\end{equation*}
where $\mathbf{u}_\D^i \!=\! \bar{u}_{i:i+L-1}$ and $\mathbf{z}_\D^i \!=\! \bar{z}_{i:i+L-1}$ for $i\in[\ell]$. 
To ensure that the data is sufficiently informative to represent the nonlinear system, we impose the following condition \citep{shang2024willems}.

\vspace{-2mm}

\begin{defn}[Lifted excitation]
\label{def:excitation}
We say that $H_\D$ provides \emph{lifted excitation of order $L$} if the matrix
$
\tiny 
\begin{bmatrix}
\mathbf{u}_\D^1 & \cdots & \mathbf{u}_\D^l \\
\psi(z_1^1) & \cdots & \psi(z_1^l)
\end{bmatrix}
\in \mathbb{R}^{(n_uL + n_x) \times l}
$
has full row rank, where $z_1^i \in \mathbb{R}^{n_z}$ denotes the first element of $\mathbf{z}_\D^i$.
\end{defn}
\vspace{-2mm}

Although the Koopman embedding is unknown, its dimension $n_x$ can be identified from sufficiently long data \citep{markovsky2022identifiability}, and an upper bound often suffices in practice.
For Koopman-linearizable systems, a data matrix $H_\D$ satisfying lifted excitation is guaranteed to exist and can be constructed by designing $\mathbf{u}_\D$ \citep{shang2024willems}. 
With lifted excitation and an initial trajectory longer than $n_x$, we are ready to present the model-free linear representation.

\vspace{-1mm}

\begin{thm}[Model-free representation \citep{shang2024willems}]
\label{thm:dd-k}
    Consider a Koopman-linearizable nonlinear system. Suppose $H_\D$ satisfies lifted excitation of order $L$. At time $t$, we collect the most recent input-state sequence $\mathbf{u}_{\ini, t}, \mathbf{z}_{\ini, t}$ of length $T_\ini \ge n_x$. Then, the concatenated sequence
    $[\mathbf{u}_{\ini}^\tr, \mathbf{u}_\f^\tr, \mathbf{z}_{\ini}^\tr, \mathbf{z}_\f^\tr]^\tr$, where $\mathbf{u}_\f \in \mathbb{R}^{n_u W}, \mathbf{z}_\f \in \mathbb{R}^{n_zW}$, is a valid length-$L$ trajectory if and only if there exists $g \in \mathbb{R}^{l}$ such that $H_\D g = [\mathbf{u}_{\ini}^\tr, \mathbf{u}_\f^\tr, \mathbf{z}_{\ini}^\tr, \mathbf{z}_\f^\tr]^\tr.$
    \vspace{-1mm}
\end{thm}
\vspace{-2mm}

To apply the above representation in \eqref{eq:DDPC}, we set $\mathbf{u}_\f$, $\mathbf{z}_\f$ as decision variables $u_{1:W|t}, z_{1:W|t}$.
By Theorem \ref{thm:dd-k}, the feasible set of $(u_{1:W|t}, z_{1:W|t})$ includes all trajectories consistent with $(\mathbf{u}_{\ini,t}, \mathbf{z}_{\ini,t})$, directly implying the following policy equivalence with \eqref{eq:nonlinear-MPC} and \eqref{eq:linear-MPC}.
\vspace{-1mm}
\begin{cor} \label{cor:policy-equivalence}
    Let $\pi_t^{\textnormal{DDPC}}(z_t;\mathbf r)$ denote the online policy obtained by solving \eqref{eq:DDPC} in Algorithm~\ref{alg:MPC}. Then $\pi_t^{\textnormal{DDPC}}(z_t;\mathbf r)=\pi_t^{\textnormal{N-MPC}}(z_t;\mathbf r)=\pi_t^{\textnormal{L-MPC}}(x_t;\mathbf r)$ for all $t\in[T]$ where $x_t=\psi(z_t)$.
\end{cor}

\vspace{-5mm}
\section{Performance Guarantees for Predictive Tracking} \label{sec:regret-analysis}

\vspace{-2mm}

We demonstrate the effectiveness of Algorithm~\ref{alg:MPC}, in particular \eqref{eq:DDPC}, by establishing a dynamic regret bound and outlining the main analysis strategy. By Corollary \ref{cor:policy-equivalence}, the three MPC formulations in Algorithm~\ref{alg:MPC} produce identical control actions given the same target trajectory. We therefore denote the common policy by $\pi^{\textnormal{MPC}}\!=\!\{\pi_t^{\textnormal{MPC}}\}_{t=1}^T$ with resulting control input $\mathbf{u}^{\textnormal{MPC}}\!=\!\{u_t^{\textnormal{MPC}}\}_{t=1}^T$. Specifically, our analysis relies on the expression for $\pi_t^{\textnormal{L-MPC}}(x_t; \mathbf{r})$ in \eqref{eq:MPC-LQT}.

\vspace{-2mm}
\subsection{Dynamic regret bound}
\vspace{-1mm}
Before presenting the regret bound, we introduce several key system-dependent quantities used in the analysis.
Under Assumption~\ref{assump:LQ}, there exists a unique $P_\infty\succeq0$ satisfying the discrete algebraic Riccati equation (DARE),
$
P_\infty = Q + A^\tr P_\infty A - A^\tr P_\infty B (R + B^\tr P_\infty B)^{-1} B^\tr P_\infty A,
$
with the optimal gain $K_\infty\!:=\! (R+B^\tr P_{\infty}B)^{-1}B^\tr P_\infty A$ and closed-loop matrix $A_{\mathrm{cl},\infty}\!:=\!A-BK_\infty$ \citep{bertsekas2012dynamic,zhou1996robust}.

The first quantity $\rho_\infty \in (0,1)$ denotes the exponential rate at which the Riccati recursion converges to $P_\infty$ \citep{hager1976convergence}.
Second, the factor ${\gamma}_\infty := \frac{1}{2}\left(1+\rho(A_{\mathrm{cl},\infty})\right) \in (0,1)$ captures the closed-loop stability of state transition matrices $A_{\mathrm{cl},t_1\to\ t_2}$. Finally, the stabilizing window $\Delta_{\mathrm{stab}} := \mathcal{O}(\log\left(1 - \rho(A_{\mathrm{cl},\infty})\right)^{-1})$ represents the minimal horizon length for the system to exhibit stability.
All these quantities depend only on the lifted system parameters $(A, B, Q, R)$ and are independent of the full horizon length $T$.
Throughout, the $\mathcal{O}(\cdot)$ notation hides constants that depend only on $(A, B, Q, R, D_\psi)$ but not on $T$.

\vspace{-3mm}
\begin{thm}[Dynamic regret] \label{thm:regret-MPC}
Let $\lambda_\infty \!:=\! \max\{\rho_\infty, {\gamma}_\infty\}$ and suppose that the prediction horizon satisfies $W \!\geq\! \Delta_{\mathrm{stab}}$. Under Assumptions~\ref{assump:koopman} and~\ref{assump:LQ}, the dynamic regret of Algorithm~\ref{alg:MPC} satisfies
\[
\mathrm{Reg}_T(\pi^{\textnormal{MPC}}) = \mathcal{O}(W^2\lambda_\infty^{2W} T)\footnote{Hidden constants in $\mathcal{O}(\cdot)$ also do not depend on \( W \) throughout entire analysis. Moreover, this bound can be tightened to $\mathcal{O}(\rho_{\infty}^{2W}T)$ if $\max\{\rho_\infty, {\gamma}_\infty\}=\rho_\infty$; see Appendix~\ref{app:proof-improved-bound}.}.
\]
\end{thm}

\vspace{-2mm}
To the best of our knowledge, Theorem~\ref{thm:regret-MPC} provides the first dynamic regret bound for online control of nonlinear systems that admit an exact Koopman embedding.
The result shows that the regret grows linearly with the full horizon $T$, but decays exponentially with the prediction horizon $W$.
In particular, a prediction horizon of length $W = \Theta(\log T)$ suffices to achieve constant $\mathcal{O}(1)$ regret, paralleling the results in \citep{yu2020power,zhang2021regret,lin2021perturbation} for linear systems.
Notably, this performance guarantee is obtained without requiring a terminal cost or knowledge of the dynamics, relying instead on a sufficiently long prediction horizon to ensure stability.
This feature is unique to Koopman-linearizable systems due to the lifting process. 

\vspace{-3mm}
\subsection{Dynamic regret analysis}
\vspace{-1mm}
We here provide a proof sketch of Theorem~\ref{thm:regret-MPC}, with full details deferred to Appendix~\ref{app:regret-proof}.
Our analysis builds on dynamic regret results for linear MPC with predictions but differs in two crucial aspects due to the Koopman lifting structure: (i) the lifted stage cost matrix $Q$ is only positive semidefinite, and (ii) our MPC formulation lacks a terminal cost.

The first key step uses Corollary~\ref{coro:reg-eq} to reduce the analysis to the lifted linear system:
\[
{\mathrm{Reg}}_T(\pi^{\textnormal{MPC}}) = {\mathrm{Reg}}_T^{\mathrm{lft}}(\pi^{\textnormal{MPC}}) =  {J}_T^{\mathrm{lft}}(\mathbf{u}^{\textnormal{MPC}};\mathbf{r}) - (J_T^{\mathrm{lft}})^*,
\]
where ${J}_T^{\mathrm{lft}}(\mathbf{u}^{\textnormal{MPC}};\mathbf{r})$ is the cost incurred by $u_{t}^{\textnormal{MPC}}$ from Algorithm~\ref{alg:MPC}.
By the cost difference lemma \citep{kakade2003sample}, the dynamic regret can be expressed as a cumulative sum of control deviations.

\vspace{-2mm}
\begin{lem}
\label{lemma:regret-formula}
    The lifted dynamic regret can be expressed as
    $
    \mathrm{Reg}_T^{\mathrm{lft}}(\pi^{\textnormal{MPC}}) = \textstyle\sum_{t=1}^{T} \|u_t^{\textnormal{MPC}} - u_t^* \|^2_{\Sigma_t},
    $
    where $u_t^{\textnormal{MPC}}=\pi_t^{\textnormal{MPC}}(x_t; \mathbf{r})$ and $u_t^*=\pi_t^*(x_t; \mathbf{r})$ denote the control actions generated by Algorithm~\ref{alg:MPC} and the optimal offline policy, respectively, both evaluated on the lifted \textit{MPC trajectory} $x_t$.
\end{lem}

Let $w_t:=A \psi(r_t) - \psi(r_{t+1})$. From Theorems~\ref{thm:opt-offline-tracking} and \ref{thm:opt-MPC}, it is clear that the control deviation vanishes for $T-W < t \leq T$, and admits the following decomposition for $t\in[T-W]$:
\begin{align*}
    u_t^{\textnormal{MPC}} - u_t^* =& (K_t - \bar{K}_1) \left(x_t - \psi(r_t)\right) + \textstyle\sum_{i=t}^{t+W-2} (K_{t\to i} - \bar{K}_{1\to i-t+1}) w_i + \textstyle\sum_{i=t+W-1}^{T-1} K_{t\to i} w_i.
\end{align*}
With the above expression and Lemma~\ref{lemma:regret-formula}, we obtain the following regret bound decomposition
\begin{align}
    \mathrm{Reg}_T^{\mathrm{lft}}(\pi^{\textnormal{MPC}}) \leq & \underbrace{\left( \textstyle\sum_{t=1}^{T-W} \left\|\textstyle\sum_{i=t+W-1}^{T-1} K_{t\to i} w_i\right\|_{\Sigma_t}^2\right)}_{\text{truncation deviation}} + \underbrace{\left(\textstyle\sum_{t=1}^{T-W} \left\|(K_t - \bar{K}_1) \left(x_t - \psi(r_t)\right)\right\|_{\Sigma_t}^2\right)}_{\text{feedback deviation}}\nonumber \\
    &\quad + \underbrace{\left(\textstyle\sum_{t=1}^{T-W} \left\|\sum_{i=t}^{t+W-2} (K_{t\to i} - \bar{K}_{1\to i-t+1}) w_i \right\|_{\Sigma_t}^2\right)}_{\text{feedforward deviation}}.  \label{eq:bound-decomposition}
\end{align}

We now bound the regret term-by-term, with all technical details deferred to Appendices~\ref{app:LQ-property} and \ref{app:regret-proof}.
All the results below are stated with the same settings in Theorem~\ref{thm:regret-MPC}.
We first establish the exponential decay behavior of some key quantities. 

\begin{lem}[Exponential decay of state transition matrices and feedforward gains] \label{lem:STM-K-decay}
For all $t_1 < t_2 < T$, we have $\|A_{\mathrm{cl},t_1 \to t_2}\| = \mathcal{O}({\gamma}_\infty^{t_2 - t_1})$ and $\|K_{t_1 \to t_2}\| = \mathcal{O}({\gamma}_\infty^{t_2 - t_1})$.
\end{lem}

\vspace{-1mm}
The proof leverages the strong stability \citep{cohen2018online} of the closed-loop matrix $A_{\mathrm{cl},\infty}$ together with the exponential convergence of the Riccati recursion; see details in \cref{subsec:proof-lem3}. This result directly yields a bound on the truncation deviation term in the regret decomposition.
\begin{prop} \label{prop:truncation-err}
    The truncation deviation term in \eqref{eq:bound-decomposition} satisfies $(\text{truncation deviation})= \mathcal{O}({\gamma}_\infty^{2W}T).$
\end{prop}

Next, to handle the remaining two terms in the regret decomposition, we establish exponential decay bounds for the gain deviations and ensure boundedness of the lifted MPC state trajectory.
\begin{lem}[Exponential decay of feedback and feedforward deviation] \label{lemma:mat-approx-err}
For $t\in[T-W]$ and $t\leq i\leq t\!+\!W\!-\!2$, we have $\|K_t - \bar{K}_1\| = \mathcal{O}(\rho_\infty^W)$ and $\|K_{t \to i} - \bar{K}_{1 \to i-t+1}\| = \mathcal{O}(\gamma_\infty^{i-t} \rho_\infty^{W - i + t})$.
\end{lem}

The proof of bounding gain deviations is more delicate and involved, but builds on similar techniques used in the proof of Lemma \ref{lem:STM-K-decay}; see Appendices~\ref{subsec:lem-for-lem4} and \ref{subsec:proof-lem4} for details.

The bound below ensures the lifted MPC state remains bounded; see Section \ref{subsec:proof-lem5} for proof.
\begin{lem}[Uniform bound on lifted MPC states] \label{lem:bound-MPC-state}
    Suppose $W \geq \Delta_{\mathrm{stab}}$. There exists a constant \( D_x > 0 \), independent of \( T \), such that for all $t\in[T-W]$, we have $\|x_t - \psi(r_t)\| \leq D_x.$
\end{lem}

The key idea in Lemma \ref{lem:bound-MPC-state} is to leverage a sufficiently long prediction horizon. This contrasts with standard MPC, which often relies on a designed terminal cost to guarantee stability \citep{yu2020power,zhang2021regret}. 
Relying on Lemmas \ref{lemma:mat-approx-err} and \ref{lem:bound-MPC-state}, we obtain the following bounds.
\begin{prop} \label{prop:gain-deviation}
    Suppose $W \geq \Delta_{\mathrm{stab}}$. In \eqref{eq:bound-decomposition}, we have
    $\text{(feedback gain deviation)} = \mathcal{O}(\rho_\infty^{2W} T)$ and $\text{(feedforward gain deviation)} = \mathcal{O}(W^2 \lambda_\infty^{2W} T)$.
\end{prop}

It is now clear that combining Propositions~\ref{prop:truncation-err} and \ref{prop:gain-deviation} yields the desired bound in Theorem~\ref{thm:regret-MPC}. 

\vspace{-3mm}
\section{Numerical Experiments} \label{sec:experiment}

\vspace{-1mm}
\begin{figure}[t]
    \centering
    \includegraphics[width=0.31\textwidth]{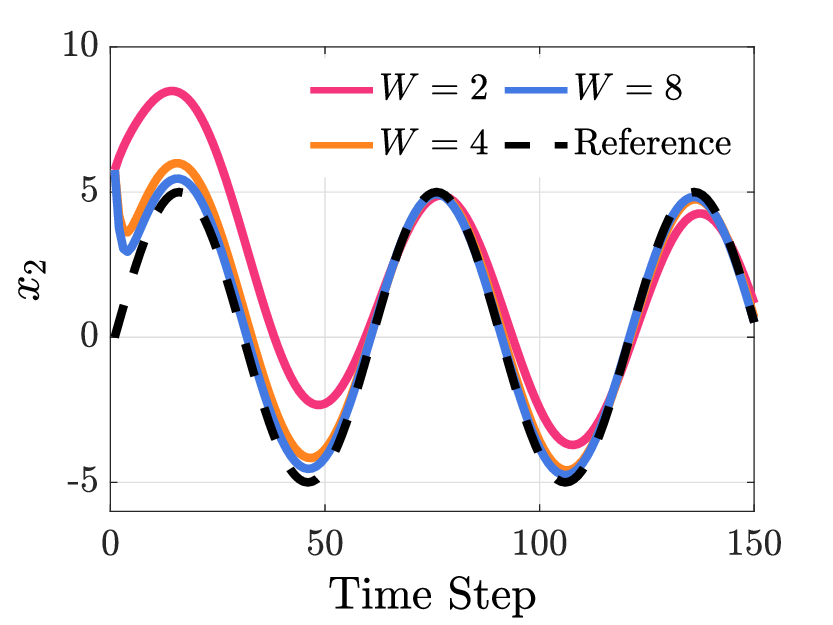} \label{Fig:Lin_track} \hspace{1.5mm}
    \includegraphics[width=0.31\textwidth]{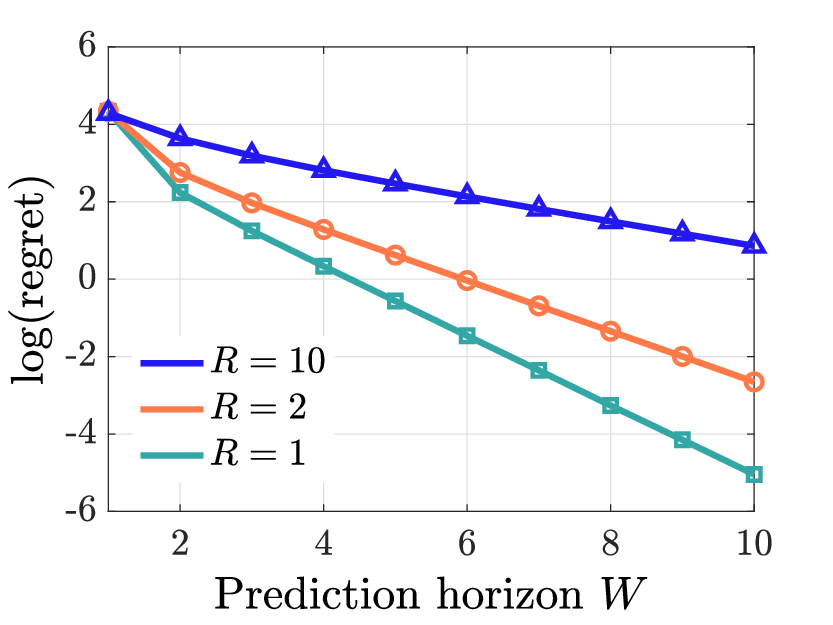} \label{Fig:Lin_regret} \hspace{1.5mm}
    \includegraphics[width=0.31\textwidth]{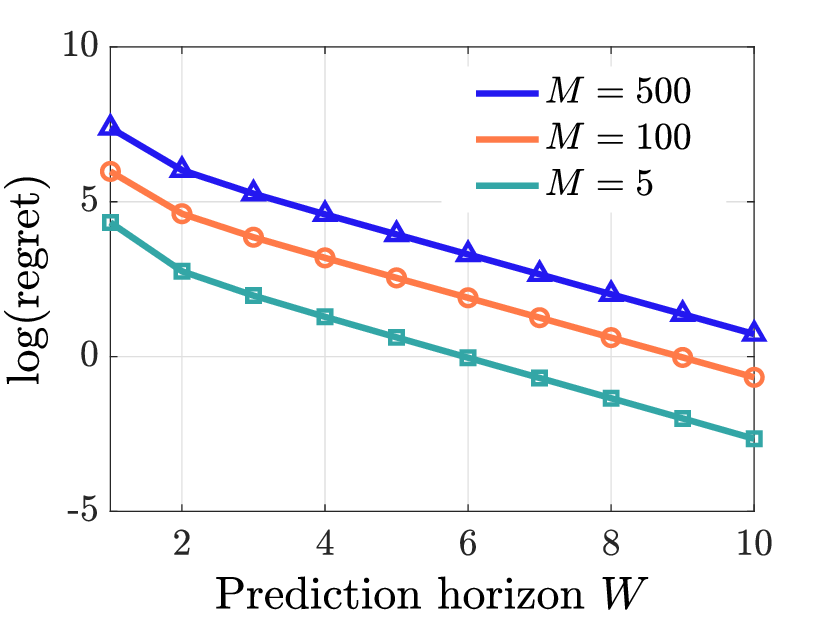} \label{subfig:regret_ref}
    \vspace{-2mm}
    \caption{Performance of Algorithm~\ref{alg:MPC} with \eqref{eq:DDPC} under varying prediction horizon $W$: tracking and reference trajectories (left), dynamic regret for different $R$ (middle) and different $M$ (right).}
    \vspace{-2mm}
    \label{fig:regret}
\end{figure}

\begin{wrapfigure}[3]{r}{.48\textwidth}
\vspace{-6.5mm}
\begin{equation}
\label{eq:slow-manifold-exp}
\begingroup
    \setlength\arraycolsep{3pt}
\def\arraystretch{0.85} 
\begin{bmatrix}
    z_{1,t+1} \\
    z_{2,t+1}
    \end{bmatrix}\!=\!
    \begin{bmatrix}
    0.99z_{1,t} \\
    0.9z_{2,t}\! +\! z_{1,t}^2\! +\! z_{1, t}^3\! +\! z_{1, t}^4\! +\! u_t
    \end{bmatrix}
\endgroup
\end{equation}
\end{wrapfigure}
We validate Theorem~\ref{thm:regret-MPC} by evaluating the tracking performance of \eqref{eq:DDPC} on the nonlinear system \eqref{eq:slow-manifold-exp}, which similar to Example~\ref{exmp:slow-manifold} is Koopman-linearizable with the lifted state $x:=[
    z_1, z_2, z_1^2, z_1^3, z_1^4]^\tr$.
Here, the control objective is to let the second state $z_2$ track a sinusoidal reference $r_{2,t}=M \sin(\pi t/30)$.
In the experiment, we set the full horizon $T=200$, state cost matrix $Q_z=\mathrm{diag}(0, 1)$, and initial trajectory length $T_\ini=10$.
For construction of the data library $H_\D$, a single trajectory of length $2W+24$ is collected for each prediction horizon $W$ (the minimum length to satisfy the lifted excitation), where the control input is uniformly sampled from $[-1,1]$.

We show in Figure~\ref{fig:regret} the tracking performance and dynamic regret. In the left subfigure, we observe that a longer prediction horizon $W$ yields smaller tracking error and faster convergence to the reference trajectory, as indicated by the quicker correction in the blue curve. 
The middle and right subfigures show that the dynamic regret decreases exponentially with increasing prediction horizon $W$, which aligns with our bound in Theorem~\ref{thm:regret-MPC}.
Moreover, we observe that the decay rate depends on the control cost matrix $R$ and reference magnitude $M$.

Finally, to illustrate how the proposed data-driven predictive control method extends to systems that are not Koopman-linearizable, we conduct an additional experiment on path tracking of two-wheeled robots. Due to the page limit, we put these additional experimental results in Appendix~\ref{app:exp-details}.

\vspace{-4mm}
\section{Conclusion} \label{sec:conclusion}
\vspace{-1mm}
This paper studied the online performance of a data-driven predictive tracking method for a class of Koopman-linearizable nonlinear systems.
The approach relied only on sufficiently rich offline data and required no explicit model identification or knowledge of Koopman lifting functions.
Leveraging the lifted representation, we established a dynamic regret bound that quantifies the tracking performance relative to the optimal offline control sequence and reveals the impact of the prediction horizon on the dynamic regret.
Numerical experiments supported our theoretical findings.
Future directions include extending the data-driven method and regret analysis to settings with disturbances or approximate Koopman embeddings, relaxing data richness conditions, incorporating imperfect predictions, and exploring alternative regret measures.

\vspace{-3mm}

\acks{This work is supported by NSF CMMI Award 2320697 and NSF CAREER Award 2340713.}

\bibliography{mylib}


\newpage
\appendix

\section{The Optimal Noncausal Policy for Linear Quadratic Tracking (LQT)} \label{app:proof-opt-offline-policy}

This section derives the optimal offline (noncausal) policy given in Theorem~\ref{thm:opt-offline-tracking} and characterizes its value function (cost-to-go). Recall $\ell_{\mathrm{lft}}(x_t,u_t;r_t):=\|x_t - \psi(r_t)\|_{Q}^2 + \|u_t\|_R^2$ for $t\in[T]$ and the LQT problem \eqref{eq:LQT} given full knowledge of the target trajectory $\mathbf{r}=r_{1:T}$
\begin{align*}
    \min_{u_{1:T}}& \quad \textstyle\sum_{t=1}^{T} \ell_{\mathrm{lft}}(x_t,u_t;r_t) \\
    \mathrm{s.t.} & \quad x_{t+1}=Ax_t+Bu_t, \ \ x_1 \text{ given}.
\end{align*}

To formalize the optimal control policy under reference tracking, we introduce notions of the optimal value function \( V_t^* \), Q-function \( Q_t^* \), and optimal policy \( \pi_t^* \), defined as follows.

\begin{defn}[LQT optimal value function, Q-function, policy] \label{def:optimal-value-Q}
    Define the terminal quantities:
    \[
    Q_{T}^*(x,u;\mathbf{r})=\ell_{\mathrm{lft}}(x,u;r_T), \quad \pi_T^*(x;\mathbf{r})=\min_u Q_T^*(x,u;\mathbf{r})=0, \quad V_T^*(x;\mathbf{r})=\ell_{\mathrm{lft}}(x,0;r_T).
    \]
    For each $t\in[T-1]$, define recursively:
    \begin{align*}
    {Q}_t^*(x, u; \mathbf{r}) 
    &= \ell_{\mathrm{lft}}(x,u;r_t) + V_{t+1}^*(Ax + Bu; \mathbf{r}), \\
    \pi_t^*(x; \mathbf{r}) 
    &= \argmin_{u} {Q}_t^*(x, u; \mathbf{r}), \\
    V_t^*(x; \mathbf{r}) 
    &= \min_{u } {Q}_t^*(x, u; \mathbf{r}) 
    = {Q}_t^*(x, \pi_t^*(x; \mathbf{r}); \mathbf{r}).
    \end{align*}
\end{defn}

\subsection{Equivalence between LQT and LQR with disturbance}

We first establish the equivalence between LQT and linear quadratic regulator (LQR) with disturbances, which serves as a key step in the subsequent proof. By a change of variables, we define the lifted tracking error state $\tilde{x}_t$ and the disturbance sequence $w_t$ as follows 
\begin{equation} \label{eq:tilde-x-and-w}
    \tilde{x}_t:=x_t-\psi(r_t), \quad w_t:=A\psi(r_t)-\psi(r_{t+1}),
\end{equation}
we can readily verify that the LQT problem \eqref{eq:LQT} is equivalent to LQR with disturbance $\mathbf{w}=w_{1:T\!-\!1}$
\begin{equation} \label{eq:LQR-with-disturbance}
    \begin{aligned}
        \min_{u_{1:T}}& \quad \textstyle\sum_{t=1}^{T} \tilde{x}_t^\tr {Q} \tilde{x}_t + u_t^\tr R u_t \\
        \mathrm{s.t.} & \quad \tilde{x}_{t+1}=A\tilde{x}_t+Bu_t+w_t,\ \ \tilde{x}_1 \text{ given}.
    \end{aligned}
\end{equation}

To formalize the optimal control policy under this equivalent LQR setting, we define the corresponding value functions and policies in terms of $\tilde{x}_t$ and $w_t$, mirroring the structure in Definition~\ref{def:optimal-value-Q}. Under this change of variables \eqref{eq:tilde-x-and-w}, we note that $\ell_{\mathrm{lft}}({x}, u; r_t)=\ell_{\mathrm{lft}}(\tilde{x}, u; 0)=\tilde{x}^\tr {Q} \tilde{x} + u^\tr R u$.

\begin{defn}[LQR optimal value function, Q-function, policy]
Define the terminal quantities:
\[
    \tilde{Q}_T^*(\tilde{x}, u; \mathbf{w}) := \ell_{\mathrm{lft}}(\tilde{x}, u; 0), \quad \tilde{\pi}_T^*(\tilde{x};\mathbf{w})=\min_u \tilde{Q}_T^*(\tilde{x},u;\mathbf{w})=0
    \quad
    \tilde{V}_T^*(\tilde{x}; \mathbf{w}) := \ell_{\mathrm{lft}}(\tilde{x}, 0; 0).
\]
For each \( t \in [T - 1] \), define recursively:
\begin{align*}
    \tilde{Q}_t^*(\tilde{x}, u; \mathbf{w}) 
    &:= \ell_{\mathrm{lft}}(\tilde{x}, u; 0) +  \tilde{V}_{t+1}^*(A \tilde{x} + B u + w_t; \mathbf{w}), \\
    \tilde{\pi}_t^*(\tilde{x}; \mathbf{w}) 
    &:= \argmin_u  \tilde{Q}_t^*(\tilde{x}, u; \mathbf{w}), \\
    \tilde{V}_t^*(\tilde{x}; \mathbf{w}) 
    &:= \min_{u }  \tilde{Q}_t^*(\tilde{x}, u; \mathbf{w}) =  \tilde{Q}_t^*(\tilde{x},  \tilde{\pi}_t^*(\tilde{x}; \mathbf{w}); \mathbf{w}).
\end{align*}
\end{defn}

\subsection{Proof of Theorem~\ref{thm:opt-offline-tracking}}

We restate Theorem~\ref{thm:opt-offline-tracking} in terms of the lifted error state $\tilde{x}$ and the disturbance sequence $\mathbf{w}$ as follows, and provide its proof.
\begin{thm} \label{thm:opt-LQR-policy}
The optimal (noncausal) policy for \eqref{eq:LQR-with-disturbance}, given full knowledge of \( \mathbf{w} \), is given by
\[
    u_t^* = \pi_t^*(\tilde{x}_t; \mathbf{w}) = -K_t \tilde{x}_t - \textstyle\sum_{i=t}^{T-1} K_{t \to i} w_i,
\]
where the feedback gain \( K_t \) and the feedforward gain $K_{t \to i} = \Sigma_t^{-1} B^\tr (A_{\mathrm{cl}, i} \cdots A_{\mathrm{cl}, t+1})^\tr P_{i+1}$
are defined via the Riccati recursion in Definition~\ref{def:Riccati-recursion-LQT}.
Moreover, the corresponding optimal value function (cost-to-go) at time $t$ is given by
$
    \tilde{V}_t^*(\tilde{x};\mathbf{w})=\tilde{x}^\tr P_t\tilde{x}+v_t^\tr \tilde{x}+q_t,
$
where the recursive relation of $(P_t,v_t,q_t)$ is given in \eqref{eq:value-fun}. 
Equivalently, in terms of the original lifted state $x$ and $\mathbf{r}$, the value function admits the form
$
    V_t^*({x};\mathbf{r})=\left(x-\psi(r_t)\right)^\tr P_t\left(x-\psi(r_t)\right)+v_t^\tr \left(x-\psi(r_t)\right) + q_t.
$
\end{thm}
\begin{proof}
    The proof proceeds via dynamic programming, recursively computing the optimal control actions. For $t \in [T-1]$, consider the optimal value function 
    \[
    \tilde{V}_t^*(\tilde{x};\mathbf{w}) = \min_u \left( \tilde{x}^\tr Q \tilde{x} + u^\tr R u + V_{t+1}^*(A \tilde{x} + B u + w_t) \right),
    \]
    with the terminal condition $\tilde{V}_{T}^*(\tilde{x};\mathbf{w}) = \tilde{x}^\tr Q \tilde{x}$. We aim to show that for all $t\in[T]$, the value function admits the quadratic form
    $
    \tilde{V}_t^*(\tilde{x};\mathbf{w})=\tilde{x}^\tr P_t\tilde{x}+v_t^\tr \tilde{x}+q_t.
    $
    At $t=T$, the claim holds trivially with $(P_{T},v_{T},q_{T})=(Q,0,0)$. 
    We now proceed by backward induction. Suppose that 
    $
        \tilde{V}^*_{t+1}(\tilde{x})=\tilde{x}^\tr P_{t+1}\tilde{x}+v_{t+1}^\tr \tilde{x}+q_{t+1}
    $
    holds for some $t\in[T-1]$ as the induction hypothesis. Then the value function at time $t$ is
    \begin{align*}
    \small
        \tilde{V}_t^*(\tilde{x};\mathbf{w})=\min_u \Big( \tilde{x}^\tr Q \tilde{x} + u^\tr R u + (A\tilde{x}+Bu+w_t&)^\tr P_{t+1}(A\tilde{x}+Bu+w_t)\\
        &+v_{t+1}^\tr (A\tilde{x}+Bu+w_t)+q_{t+1} \Big).
    \end{align*}
    The above is a quadratic optimization problem in $u$, rewritten as follows:
    \[
    \small \tilde{V}^*_t(\tilde{x};\mathbf{w})\!=\!\min_u\left( \begin{bmatrix}
        u \\ \tilde{x} \\ w_t \\ v_{t+1}
        \end{bmatrix}^\tr \begin{bmatrix}
        R+B^\tr P_{t+1}B & B^\tr P_{t+1} A & B^\tr P_{t+1} & \frac{1}{2}B^\tr \\ A^\tr P_TB & {Q}+A^\tr P_{t+1}A & A^\tr P_{t+1} & \frac{1}{2}A^\tr \\ P_{t+1} B & P_{t+1} A & P_{t+1} & \frac{1}{2}I \\ \frac{1}{2}B & \frac{1}{2}A & \frac{1}{2}I & 0
        \end{bmatrix}\begin{bmatrix}
        u \\ \tilde{x} \\ w_t \\ v_{t+1} \end{bmatrix}+q_{t+1} \right),
    \]
    where the minimizer (optimal control action) is given by
    \begin{equation} \label{eq:opt-u_t}
        u_t^*=\tilde{\pi}_t^*(\tilde{x};\mathbf{w}) = -(R+B^\tr P_{t+1}B)^{-1}B^\tr \left( P_{t+1}A\tilde{x} + P_{t+1}w_t + \frac{1}{2} v_{t+1} \right).
    \end{equation}
    Substituting \eqref{eq:opt-u_t} back into the value function yields
    \begin{align*}
        \tilde{V}^*_t(\tilde{x};\mathbf{w})=&\begin{bmatrix}
         \tilde{x} \\ w_t \\ v_{t+1}
        \end{bmatrix}^\tr \begin{bmatrix}
              {Q}+A^\tr P_{t+1}A & A^\tr P_{t+1} & \frac{1}{2}A^\tr \\  P_{t+1} A & P_{t+1} & \frac{1}{2}I \\  \frac{1}{2}A & \frac{1}{2}I & 0
        \end{bmatrix}\begin{bmatrix}
            \tilde{x} \\ w_t \\ v_{t+1}
        \end{bmatrix}\\
        &-\begin{bmatrix}
            \tilde{x} \\ w_t \\ v_{t+1}
        \end{bmatrix}^\tr\begin{bmatrix}
                A^\tr P_{t+1}B \\ P_{t+1}B \\ \frac{1}{2}B
            \end{bmatrix}(R+B^\tr P_{t+1}B)^{-1}\begin{bmatrix}
                A^\tr P_{t+1}B \\ P_{t+1}B \\ \frac{1}{2}B
            \end{bmatrix}^\tr\begin{bmatrix}
            \tilde{x} \\ w_t \\ v_{t+1}
        \end{bmatrix} + q_{t+1},
    \end{align*}
    which is in the quadratic form $\tilde{V}^*_t(\tilde{x};\mathbf{w})=\tilde{x}^\tr P_t \tilde{x} + v_t^\tr \tilde{x} + q_t$ with the recursion given by
    \begin{equation} \label{eq:value-fun}
        \begin{aligned}
            P_t&=Q+A^\tr P_{t+1}A -A^\tr P_{t+1}B(R+B^\tr P_{t+1}B)^{-1}B^\tr P_{t+1}A,\\
            v_t^\tr&=2\begin{bmatrix}
                w_t \\ v_{t+1}
            \end{bmatrix}^\tr \left(
            \begin{bmatrix}
            P_{t+1} A \\
            \frac{1}{2} A
            \end{bmatrix}
            -
            \begin{bmatrix}
            P_{t+1} B \\
            \frac{1}{2} B
            \end{bmatrix}
            \left(R + B^\tr P_{t+1} B\right)^{-1}
            B^\tr P_{t+1} A
            \right),\\
            q_t& = \begin{bmatrix}
            w_t \\ v_{t+1}
            \end{bmatrix}^\tr
            \left(
            \begin{bmatrix}
            P_{t+1} & \frac{1}{2} I \\
            \frac{1}{2} I & 0
            \end{bmatrix}
            -
            \begin{bmatrix}
            P_{t+1} B \\
            \frac{1}{2} B
            \end{bmatrix}
            \left(R + B^\tr P_{t+1} B\right)^{-1}
            \begin{bmatrix}
            P_{t+1} B \\
            \frac{1}{2} B
            \end{bmatrix}^\tr
            \right)
            \begin{bmatrix}
            w_t \\ v_{t+1}
            \end{bmatrix}
            + q_{t+1}.
        \end{aligned}
    \end{equation}
    We then proceed to derive an explicit expression for $u_t^*$ in \eqref{eq:opt-u_t} in terms of $\mathbf{w}=w_{1:T}$. For brevity, we use the notations $\Sigma_t := R+B^\tr P_{t+1}B$, $\Sigma_t := R+B^\tr P_{t+1}B$, $A_{\mathrm{cl},t} := A-BK_t$ and
    \[
    S_t := I - P_{t+1}B\Sigma_t^{-1} B^\tr.
    \]
    Hence, $u_t^*$ in \eqref{eq:opt-u_t} is rewritten as
    \begin{equation} \label{eq:opt-u_t-2}
        u_t^* = -K_t \tilde{x} -\Sigma_t^{-1}B^\tr \left(P_{t+1}w_t+\frac{1}{2}v_{t+1}\right).
    \end{equation}
    We start with the recursive expression of $v_t$
    \begin{equation} \label{eq:v_t}
        \begin{aligned}
            \frac{1}{2}v_t=&\left( \begin{bmatrix}
                P_{t+1}A \\ \frac{1}{2}A
            \end{bmatrix}^\tr - A^\tr P_{t+1}B\Sigma_t^{-1} \begin{bmatrix}
                P_{t+1}B \\ \frac{1}{2}B
            \end{bmatrix}^\tr \right) \begin{bmatrix}
                w_t \\ v_{t+1}
            \end{bmatrix} \\
            =&A^\tr (P_{t+1}\!-\!P_{t+1}B\Sigma_t^{-1} B^\tr P_{t+1})w_t+A^\tr(I\!+\!P_{t+1}B\Sigma_t^{-1}B^\tr)v_{t+1}\\
            =&A^\tr S_t \left(P_{t+1}w_t+\frac{1}{2}v_{t+1}\right).
        \end{aligned}
    \end{equation}
    Substituting \eqref{eq:v_t} into \eqref{eq:opt-u_t-2} gives
    \begin{align*}
        u_t^* = & -K_t \tilde{x} -\Sigma_t^{-1}B^\tr \left( P_{t+1}w_t + A^\tr S_{t+1}\left( P_{t+2}w_{t+1} + \frac{1}{2}v_{t+2} \right) \right)\\
        = & -K_t \tilde{x} -\Sigma_t^{-1}B^\tr \bigg( P_{t+1}w_t + A^\tr S_{t+1} P_{t+2}w_{t+1} \\
        &\qquad \qquad \qquad \qquad \qquad + A^\tr S_{t+1} A^\tr S_{t+2} P_{t+3}w_{t+2} +  \frac{1}{2}A^\tr S_{t+1} A^\tr S_{t+2} v_{t+2} \bigg)\\
        = & -K_t \tilde{x} -\Sigma_t^{-1}B^\tr \bigg( P_{t+1}w_t + A_{\mathrm{cl},t+1}^\tr P_{t+2}w_{t+1} \\
        &\qquad \qquad \qquad \qquad \qquad +A_{\mathrm{cl},t+1}^\tr A_{\mathrm{cl},t+2}^\tr P_{t+3}w_{t+2} +  \frac{1}{2}A^\tr S_{t+1} A^\tr S_{t+2} v_{t+2} \bigg),
    \end{align*}
    where the last equality follows from the identity $A^\tr S_{t}  = A_{\mathrm{cl},t}^\tr$.
    By iterative substitution and noting that $v_T=0$, we eventually obtain
    \begin{align*}
        u_t^* &=  -K_t \tilde{x} - \Sigma_t^{-1}B^\tr \bigg( P_{t+1}w_t + A_{\mathrm{cl},t+1}^\tr P_{t+2}w_{t+1} \\
        &\qquad \qquad \qquad \qquad \quad + (A_{\mathrm{cl},t+2} A_{\mathrm{cl},t+1})^\tr P_{t+3}w_{t+2} \!+\! \cdots \!+\! (A_{\mathrm{cl},T-1} \!\cdots\! A_{\mathrm{cl},t+1})^\tr P_{T}w_{T-1} \bigg)\\
        &= -K_t \tilde{x} - \sum_{i=t}^{T-1} K_{t\to i} w_i, \qquad K_{t \to i} := \Sigma_t^{-1} B^\tr (A_{\mathrm{cl}, i} \cdots A_{\mathrm{cl}, t+1})^\tr P_{i+1}.
    \end{align*}
    Lastly, the optimal offline policy in Theorem~\ref{thm:opt-offline-tracking} follows directly from the change of variables.
\end{proof}

\section{Structural Properties in Linear Quadratic Control} \label{app:LQ-property}

\subsection{Noiseless linear quadratic regulator}

We briefly review the optimal controllers for the noiseless LQR problem in both the finite- and infinite-horizon settings, which are key components for our subsequent analysis.
See \cite{zhou1996robust,bertsekas2012dynamic} for more details.
We begin with the finite-horizon problem:
\begin{equation} \label{eq:FH-LQR}
    \begin{aligned}
    \min_{u_{1:T}}& \quad \textstyle\sum_{t=1}^{T}x_t^\tr Qx_t+u_t^\tr R u_t \\
    \mathrm{s.t.} & \quad  x_{t+1}=Ax_t+Bu_t,\ \ x_1 \text{ given}.
    \end{aligned}
\end{equation}
Recall the Riccati recursion in Definition~\ref{def:Riccati-recursion-LQT}. The solution to \eqref{eq:FH-LQR} is given below.
\begin{prop}
    The optimal control policy for \eqref{eq:FH-LQR} is given by $u_t=-K_tx_t$ for all $t\in[T]$. Moreover, the optimal cost-to-go is quadratic in the state, i.e., for any $t\in[T]$ and $x_t$,
    \[
    x_t^\tr P_t x_t = \min_{u_{t:T}}\textstyle\sum_{i=t}^{T}x_i^\tr Q x_i + u_i^\tr R u_i.
    \]
\end{prop}

We now turn to the infinite-horizon case, where the optimal solution is characterized by the discrete algebraic Riccati equation (DARE).
\begin{prop} \label{prop:DARE}
    There exists a unique positive semidefinite matrix $P_\infty\succeq0$ that solves the DARE:
    \begin{equation} \label{eq:DARE}
        P_\infty=Q+A^\tr P_{\infty}A -A^\tr P_{\infty}B(R+B^\tr P_{\infty}B)^{-1}B^\tr P_{\infty}A.
    \end{equation}
\end{prop}
Given the solution $P_\infty$, we define the following associated steady-state quantities:
\[
\Sigma_\infty := R+B^\tr P_{\infty}B,\quad K_\infty := \Sigma_\infty^{-1}B^\tr P_\infty A, \quad A_{\text{cl},\infty}:=A-BK_\infty.
\]
Note that these are intrinsic control-theoretic quantities determined solely by $(A,B,Q,R)$. 

The following result describes the infinite-horizon optimal policy and its stability.

\begin{prop} \label{prop:opt-inf-LQR}
    As $T\to\infty$, the optimal control policy for \eqref{eq:FH-LQR} converges to the stationary policy $u_t=-K_\infty x_t$.
    Moreover, the corresponding closed-loop matrix $A_{\text{cl},\infty}$ is stable, i.e., $\rho(A_{\text{cl},\infty})<1$.
\end{prop}

\subsection{Strong stability and regularity of Riccati-based quantities}
\label{subsec:property-offline-policy}

We establish some structural properties of key control-theoretic quantities intrinsic to $(A,B,Q,R)$ and the Riccati recursion, which underpin the analysis of the lifted-space optimal offline and online MPC controllers in Theorems~\ref{thm:opt-offline-tracking} and \ref{thm:opt-MPC}.
We start with a quantitative notion of stability.

\begin{defn}[Strong Stability \citep{cohen2018online}]
A matrix $A-BK$ is said to be \emph{$(\kappa, \gamma)$ strongly stable} for some $\kappa > 0$ and  $\gamma \in [0,1)$ if there exist matrices $H$ and $L$ such that
\[
A - BK = H L H^{-1}, \quad \|H\| \cdot \|H^{-1}\| \leq \kappa, \quad \|L\| \leq \gamma.
\]
\end{defn}
We may also refer to a gain matrix $K$ as (strongly) stable, meaning that its corresponding closed-loop matrix $A-BK$ is (strongly) stable.
Moreover, any stable matrix is strongly stable for some $\kappa$ and $\gamma$.

\begin{lem}\cite[Lemma B.1.]{cohen2018online} \label{lem:stable-imply-SS}
Suppose that \( A - BK \) is stable, i.e., \( \rho(A - BK) < 1 \). Then there exist constants \( \kappa > 0 \) and \( \gamma \in [0,1) \) such that \( A - BK \) is \( (\kappa, \gamma) \) strongly stable.
\end{lem}

Below we establish the strong stability of the closed-loop matrix $A_{\mathrm{cl},\infty} \!:=\! A - B K_\infty$.

\begin{cor} \label{coro:Acl-SS}
    There exists some $\kappa_{\infty}>0$ and $\tilde{\gamma}_{\infty}\in[0,1)$ such that $A_{\mathrm{cl},\infty}$ is $(\kappa_{\infty},\tilde{\gamma}_{\infty})$ strongly stable. That is, there exist matrices $H_{\mathrm{cl},\infty}$ and $L_{\mathrm{cl},\infty}$ such that 
    \[
    A_{\mathrm{cl},\infty} = H_{\mathrm{cl},\infty} L_{\mathrm{cl},\infty} H_{\mathrm{cl},\infty}^{-1},
    \quad
    \|H_{\mathrm{cl},\infty}\| \cdot \|H_{\mathrm{cl},\infty}^{-1}\| \leq \kappa_{\infty},
    \quad
    \|L_{\mathrm{cl},\infty}\| \leq \tilde{\gamma}_{\infty}.
    \]
\end{cor}

\begin{proof}
    Since $A_{\mathrm{cl},\infty}$ is stable by Proposition~\ref{prop:opt-inf-LQR}, Lemma~\ref{lem:stable-imply-SS} implies that it is strongly stable.
\end{proof}

In fact, from the proof of Lemma~\ref{lem:stable-imply-SS}, we can choose $\tilde{\gamma}_{\infty}:=\rho(A_{\mathrm{cl},\infty})\in[0,1)$.

\begin{lem}
    For any $i\geq0$, we have $\|A_{\mathrm{cl},\infty}^i\|\leq \kappa_{\infty} \tilde{\gamma}_{\infty}^i$.
\end{lem}
\begin{proof}
    By Corollary~\ref{coro:Acl-SS}, it follows that $\|A_{\mathrm{cl},\infty}^i\|=\|H_{\mathrm{cl},\infty} L_{\mathrm{cl},\infty}^i H_{\mathrm{cl},\infty}^{-1}\|\leq \kappa_{\infty} \tilde{\gamma}_{\infty}^i$.
\end{proof}

We next recall a classical result from \cite{hager1976convergence}, which guarantees exponential convergence of the Riccati recursion to the DARE solution $P_\infty$.
This result requires only $Q\succeq0$ and detectability of $(A,Q)$, unlike more recent analyses (e.g., \cite{dean2018regret,foster2020logarithmic,zhang2021regret}) that assume $Q\succ0$ to obtain explicit rates.
Since the Koopman lifting yields only $Q\succeq0$, we rely on this more general guarantee. 

\begin{prop}[\cite{hager1976convergence}] \label{prop:riccati-gain-convergence}
Let $P_t$ and $K_t$ be defined as in Definition~\ref{def:Riccati-recursion-LQT}. Then there exists a constant \( \rho_\infty \in (0,1) \), independent of the horizon length $T$, such that for all $t\in[T]$,
\[
\|P_t - P_\infty\| = \mathcal{O}(\rho_\infty^{T - t})
\quad \text{and} \quad
\|K_t - K_\infty\| = \mathcal{O}(\rho_\infty^{T - t}),
\]
where the hidden constant in \( \mathcal{O}(\cdot) \) depends only on \( (A, B, Q, R) \) and not on \( T \).
\end{prop}

To facilitate the analysis of time-varying matrices $A_{\mathrm{cl},t}$, we also define the transformed versions $L_{\mathrm{cl},t}:=H_{\mathrm{cl},\infty}^{-1} A_{\mathrm{cl},t}H_{\mathrm{cl},\infty}$.
We next show that $L_{\mathrm{cl},t}$ remain uniformly bounded by $1$ for all but the last $\Delta_{\mathrm{stab}}$ steps of the horizon $T$, where $\Delta_{\mathrm{stab}}$ is independent of $T$.

\begin{lem} \label{lem:bound-L}
    Let ${\gamma}_\infty:=\frac{1}{2}(1+\tilde{\gamma}_{\infty})$ and $\Delta_{\mathrm{stab}}:=\mathcal{O}(\log(1-\tilde{\gamma}_{\infty})^{-1})$, where the hidden constant is independent of $T$. Then it holds that $\|L_{\mathrm{cl},t}\|\leq {\gamma}_\infty<1$ for all $t\leq T_{\mathrm{stab}}:=T-\Delta_{\mathrm{stab}}$.
\end{lem}

\begin{proof}
    For any $t\in[T]$, we have $\|L_{\mathrm{cl},t}\|\leq \|L_{\mathrm{cl},\infty}\| + \|L_{\mathrm{cl},t}-L_{\mathrm{cl},\infty}\| \leq \tilde{\gamma}_{\infty}+\kappa_\infty \|A_{\mathrm{cl},t}-A_{\mathrm{cl},\infty}\| \leq \tilde{\gamma}_\infty+\kappa_\infty \|B\| \cdot \|K_t-K_\infty\|.$
    By Proposition~\ref{prop:riccati-gain-convergence}, $\|K_t-K_\infty\|=\mathcal{O}(\rho_\infty^{T-t})$. Thus, there exists a constant $\Delta_{\mathrm{stab}}=\mathcal{O}(\log(1-\tilde{\gamma}_\infty)^{-1})$, independent of $T$, such that the desired result holds.
\end{proof}

To facilitate the MPC analysis, we introduce the following Riccati recursion.
\begin{defn}[Riccati recursion for MPC] \label{def:Riccati-recursion-MPC}
Define the backward Riccati recursion as:
\begin{align*}
    \bar{P}_{W}:=Q,\quad \bar{P}_t &:= {Q} + A^\tr \bar{P}_{t+1} A - A^\tr \bar{P}_{t+1} B \bar{\Sigma}_t^{-1} B^\tr \bar{P}_{t+1} A, \quad \bar{K}_t := \bar{\Sigma}_t^{-1} B^\tr \bar{P}_{t+1} A
\end{align*}
for $t\in[W-1]$ where we define $\bar{\Sigma}_t := R + B^\tr \bar{P}_{t+1} B$.
\end{defn}
Note that the above time indices for MPC correspond to those in Definition~\ref{def:Riccati-recursion-LQT} via $t \mapsto T - W + t - 1$.
That is, $\bar{K}_t = K_{T-W+t-1}$, $\bar{P}_{t+1} = P_{T-W+t}$, and $\bar{\Sigma}_t = \Sigma_{T-W+t-1}$ for all $t \in [W]$.
We also define 
\[
\bar{A}_{\mathrm{cl},t} := A - B \bar{K}_t,\quad \bar{A}_{\mathrm{cl},1\to i-t+1}:=\bar{A}_{\mathrm{cl},i - t + 1} \cdots \bar{A}_{\mathrm{cl},2}
\]
for $t\leq i$ with $t\in[W]$.
Finally, we have the following bound on  $\bar{A}_{\mathrm{cl},1}={A}_{\mathrm{cl},T-W}$.

\begin{lem} \label{lem:MPC-stability}
    For any $W \geq \Delta_{\mathrm{stab}}$, we have $\|\bar{A}_{\mathrm{cl},1}^i\|\leq\kappa_\infty \gamma_\infty^i$.
\end{lem}
\begin{proof}
    From Lemma~\ref{lem:bound-L}, it follows that for $W \geq \Delta_{\mathrm{stab}}$ we have $\|L_{\mathrm{cl},T-W}\|\leq{\gamma}_\infty$. Therefore, $\|A_{\mathrm{cl},T-W}^i\|=\|H_{\mathrm{cl},\infty} L_{\mathrm{cl},T-W}^i H_{\mathrm{cl},\infty}^{-1}\|\leq \kappa_\infty{\gamma}_\infty^i$.
\end{proof}

Lemma~\ref{lem:MPC-stability} establishes a stability property of the MPC closed-loop dynamics and plays a key role in proving the boundedness of the MPC state trajectory in Lemma~\ref{lem:bound-MPC-state}.

\section{Proofs of Section~\ref{sec:regret-analysis}} \label{app:regret-proof}

\subsection{Proof of Lemma~\ref{lemma:regret-formula}}

The proof uses the cost difference lemma \citep{kakade2003sample,fazel2018global,zhang2021regret}.

\begin{lem}\label{lem:cost-diff-lemma}
For two policies \( \pi_1, \pi_2 \), the difference of their cost-to-go $V_1$ and $V_2$ is given by
\[
V_1^{\pi_2}(x) - V_1^{\pi_1}(x) = \textstyle\sum_{t=1}^{T} Q_t^{\pi_1}(x_t^{\pi_2}, u_t^{\pi_2}) - V_t^{\pi_1}(x_t^{\pi_2}) \tag{25}
\]
\emph{where $\{x_t^{\pi_2}, u_t^{\pi_2}\}$ are the trajectory generated by starting at initial state $x$ and imposing policy $\pi_2$.}
\end{lem}

\begin{proof}
    By Lemma~\ref{lem:cost-diff-lemma}, the relative cost between the MPC and offline optimal policy is equal to
    \begin{align*}
        J_T^\mathrm{lft}(\mathbf{u}^{\textnormal{MPC}};\mathbf{r})-J_T^\mathrm{lft}(\mathbf{u}^*;\mathbf{r}) &= \textstyle\sum_{t=1}^T {Q}_t^*({x}_t^{\textnormal{MPC}},{u}_t^{\textnormal{MPC}};\mathbf{r}) 
        - {Q}_t^*(x_t^{\textnormal{MPC}},{u}_t^*;\mathbf{r}).
        \label{eq:perf-diff}
    \end{align*}
    By Definition~\ref{def:optimal-value-Q} and Theorem~\ref{thm:opt-LQR-policy}, $Q_t^*(x_t^{\textnormal{MPC}},\cdot;\mathbf{r})$ is strongly convex quadratic and  $u_t^*$ is its minimizer. Therefore, the first-odrer optimality condition implies that
    \[
    Q_t^*(x_t^{\textnormal{MPC}},u_t^{\textnormal{MPC}};\mathbf{r})-Q_t^*(x_t^{\textnormal{MPC}},u_t^*;\mathbf{r})=\|u_t^{\textnormal{MPC}}-u_t^*\|^2_{\nabla_u^2 Q_t^*(x_t^{\textnormal{MPC}},u_t^*;\mathbf{r})}.
    \]
    Finally, the result follows from $\nabla_u^2Q_t^*(x_t^{\textnormal{MPC}},u_t^*;\mathbf{r})=\Sigma_t$.
\end{proof}

\subsection{Proof of Lemma~\ref{lem:STM-K-decay}} \label{subsec:proof-lem3}

We highlight the key proof idea from Lemma~\ref{lem:bound-L}: matrices $L_{\mathrm{cl},t}$ are uniformly bounded by $1$ for all but the final $\Delta_{\mathrm{stab}}$ steps, where $\Delta_{\mathrm{stab}}$ is independent of the horizon $T$. Therefore, any bounded product involving at most $\Delta_{\mathrm{stab}}$ terms near the end of the horizon can be absorbed into the constant in $\mathcal{O}(\cdot)$, ensuring that the overall product still decays exponentially in $t_2 - t_1$.

\begin{proof}
    Since $A_{\mathrm{cl},t} = H_{\mathrm{cl},\infty} L_{\mathrm{cl},t} H_{\mathrm{cl},\infty}^{-1}$, we have
    $
    A_{\mathrm{cl},t_1 \to t_2} = H_{\mathrm{cl},\infty} L_{\mathrm{cl},t_1 \to t_2} H_{\mathrm{cl},\infty}^{-1}
    $
    and therefore
    \[
    \|A_{\mathrm{cl},t_1 \to t_2}\| \leq \|H_{\mathrm{cl},\infty}\| \cdot \|H_{\mathrm{cl},\infty}^{-1}\| \cdot \|L_{\mathrm{cl},t_1 \to t_2}\| 
    \leq \kappa_\infty \|L_{\mathrm{cl},t_1 \to t_2}\|.
    \]
    It remains to bound  $\|L_{\mathrm{cl},t_1 \to t_2}\|$ in the following three cases.
    
    \noindent\textbf{Case 1: \( t_2 \leq T_{\mathrm{stab}} \).}  
    By Lemma~\ref{lem:bound-L}, we have \( \|L_{\mathrm{cl},t}\| \leq {\gamma}_\infty < 1 \) for all \( t \in [t_1+1, t_2] \). Thus
    \[
    \|L_{\mathrm{cl},t_1 \to t_2}\| \leq {\gamma}_\infty^{t_2 - t_1}.
    \]

    \noindent\textbf{Case 2: \( t_1 \leq T_{\mathrm{stab}} < t_2 \).} In this case, decompose the product at \( T_{\mathrm{stab}} \)
    \begin{align*}
         \|L_{\mathrm{cl},t_1\to t_2}\|&\leq \|L_{\mathrm{cl},t_1\to T_{\mathrm{stab}}}\| \cdot \|L_{\mathrm{cl},T_{\mathrm{stab}}\to t_2}\| \\
         &\overset{(a)}{\leq} {\gamma}_\infty^{T_{\mathrm{stab}}-t_1} \|L_{\mathrm{cl},T_{\mathrm{stab}}\to t_2}\|  \\
         &\overset{(b)}{\leq} {\gamma}_\infty^{T_{\mathrm{stab}}-t_1} {\gamma}_\infty^{t_2-T_{\mathrm{stab}}} ({\gamma}_\infty^{-1})^{\Delta_{\mathrm{stab}}} \|L_{\mathrm{cl},T_{\mathrm{stab}}\to t_2}\|\\
         &\overset{(c)}{=}\mathcal{O}({\gamma}_\infty^{t_2-t_1}),
    \end{align*}
    where $(a)$ follows from Lemma~\ref{lem:bound-L}, $(b)$ uses the fact that $t_2-T_{\mathrm{stab}}\leq\Delta_{\mathrm{stab}}$, and $(c)$ absorbs $({\gamma}_\infty^{-1})^{\Delta_{\mathrm{stab}}} \|L_{\mathrm{cl},T_{\mathrm{stab}}\to t_2}\|$ into the constant, since the product involves at most $\Delta_{\mathrm{stab}}$ terms and is uniformly bounded independently of $T$. 
    
    \noindent\textbf{Case 3: \( T_{\mathrm{stab}} < t_1 \leq t_2 \).}  Here we have
    \begin{align*}
         \|L_{\mathrm{cl},t_1\to t_2}\|
         &= {\gamma}_\infty^{t_2-t_1} ({\gamma}_\infty^{-1})^{t_2-t_1}\|L_{\mathrm{cl},t_1 \to t_2}\|\\
         &\overset{(a)}{\leq} {\gamma}_\infty^{t_2-t_1} ({\gamma}_\infty^{-1})^{\Delta_{\mathrm{stab}}}\|L_{\mathrm{cl},t_1 \to t_2}\|\\
         &\overset{(b)}{=} \mathcal{O}({\gamma}_\infty^{t_2-t_1}),
    \end{align*}
    where $(a)$ is due to $t_2-t_1\leq\Delta_{\mathrm{stab}}$ and $(b)$ absorbs $({\gamma}_\infty^{-1})^{\Delta_{\mathrm{stab}}}\|L_{\mathrm{cl},t_1 \to t_2}\|$ into the constant, since this quantity is uniformly bounded by a constant independent of $T$.

    Combining all cases yields the desired bound.
\end{proof}

\subsection{Proof of Proposition~\ref{prop:truncation-err}}

\begin{proof}
    Let \( w_i = A \psi(r_i) - \psi(r_{i+1}) \). By Assumption~\ref{assump:koopman}, \( \|w_i\| \leq D_w := (\|A\| + 1) D_\psi \). Then,
    \begin{align*}
        \textstyle\sum_{t=1}^{T-W} \left\|\textstyle\sum_{i=t+W-1}^{T-1} K_{t\to i} w_i\right\|_{\Sigma_t}^2 \leq& \textstyle\sum_{t=1}^{T-W} \|\Sigma_t\| \cdot \left( \textstyle\sum_{i=t+W-1}^{T-1} \|K_{t\to i}\| \cdot \|w_i\| \right)^2  \\
        \leq& D_w^2 \cdot \max_t \|\Sigma_t\| \cdot \textstyle\sum_{t=1}^{T-W} \left( \sum_{i=t+W-1}^{T-1} \|K_{t\to i}\| \right)^2\\
        \leq& D_w^2 \cdot \left( \|R\| + \|B\|^2 \cdot \|P_\infty\| \right) \cdot \textstyle\sum_{t=1}^{T-W} \left(\textstyle\sum_{i=t+W-1}^{T-1} \mathcal{O}(\gamma_\infty^{i-t})\right)^2,
    \end{align*}
    where the last equality follows from Lemma~\ref{lem:STM-K-decay}.
    For the inner sum, changing the summation index via $j=i-t+1$, we obtain
    \[
    \textstyle\sum_{i=t+W-1}^{T-1} \mathcal{O}(\gamma_\infty^{i - t}) = \mathcal{O}\left(\textstyle\sum_{j=W}^{T-t} \gamma_\infty^j\right) = \mathcal{O}\left( \textstyle\sum_{j=W}^{\infty} \gamma_\infty^j\right) = \mathcal{O}(\gamma_\infty^{W}).
    \]
    Therefore, taking the square and summing over $t$ gives
    \[
    \textstyle\sum_{t=1}^{T-W} \left\|\textstyle\sum_{i=t+W-1}^{T-1} K_{t\to i} w_i\right\|_{\Sigma_t}^2= \mathcal{O}(\gamma_\infty^{2W} T).
    \]
\end{proof}

\subsection{Auxiliary lemmas for Lemma~\ref{lemma:mat-approx-err}} \label{subsec:lem-for-lem4}

To prove Lemma~\ref{lemma:mat-approx-err}, we first introduce two technical lemmas (Lemmas~\ref{lem:X-diff} and \ref{lemma:X-diff-bound}) below for bounding the difference between two intermediate matrices defined as follows:
\begin{align*}
    X_{t \to i} &:= {A}_{\mathrm{cl},t\to i}^\tr P_{i+1} = \left( A_{\mathrm{cl},i} \cdots A_{\mathrm{cl},t+1} \right)^\tr P_{i+1}, \\
    \bar{X}_{1 \to i - t + 1} &:= \bar{A}_{\mathrm{cl},1\to i-t+1}^\tr \bar{P}_{i - t + 2} = \left( \bar{A}_{\mathrm{cl},i - t + 1} \cdots \bar{A}_{\mathrm{cl},2} \right)^\tr \bar{P}_{i - t + 2}.
\end{align*}

The first lemma relates the difference of $\bar{X}_{1 \to i-t+1} - X_{t \to i}$ to $\bar{P}_{j-t+2}-P_{j+1}$ for $t+1\leq j\leq i$.

\begin{lem} \label{lem:X-diff}
    For $t\in [T-W]$ and $t \leq i \leq t+W-2$,
    \begin{align*}
        \bar{X}_{1 \to i-t+1} &- X_{t \to i}=\ {A}_{\mathrm{cl},t\to i}^\tr (\bar{P}_{i-t+2}-P_{i+1}) \\
        +\sum_{j=t+1}^i &{A}_{\mathrm{cl},t\to j}^\tr (P_{T-W+j-t+2}-P_{j+1})^\tr B (R + B^\tr P_{T-W+j-t+2} B)^{-1} B^\tr \bar{X}_{j-t+1 \to i-t+1}.
    \end{align*}
\end{lem}

\begin{proof}
    We first observe the recursive structure
    \[
    X_{t \to i}=A_{\mathrm{cl},t+1}^\tr  X_{t+1 \to i}, \quad \bar{X}_{1 \to i - t + 1}=\bar{A}_{\mathrm{cl},2}^\tr  \bar{X}_{2 \to i-t+1}.
    \]
    Therefore,
    \begin{align*}
        & \bar{X}_{1 \to i-t+1} - X_{t \to i} \\
        =&(\bar{A}_{\mathrm{cl},2}^\tr-A_{\mathrm{cl},t+1}^\tr)  \bar{X}_{2 \to i-t+1}  + A_{\mathrm{cl},t+1}^\tr  (\bar{X}_{2 \to i-t+1}-X_{t+1 \to i})\\
        =&(\bar{K}_2-K_{t+1})^\tr B^\tr \bar{X}_{2 \to i-t+1} + A_{\mathrm{cl},t+1}^\tr  (\bar{X}_{2 \to i-t+1}-X_{t+1 \to i}).
    \end{align*}
    By iterative expansion using the above recursive relation, we obtain
    \begin{align}
        &\bar{X}_{1 \to i-t+1} - X_{t \to i}\notag\\
        =&(\bar{K}_2-K_{t+1})^\tr B^\tr \bar{X}_{2 \to i-t+1} + A_{\mathrm{cl},t+1}^\tr  (\bar{X}_{2 \to i-t+1}-X_{t+1 \to i})\notag\\
        =&(\bar{K}_2-K_{t+1})^\tr B^\tr \bar{X}_{2 \to i-t+1} \notag \\
        &+ A_{\mathrm{cl},t+1}^\tr  (\bar{K}_3-K_{t+2})^\tr B^\tr \bar{X}_{3 \to i-t+1} + A_{\mathrm{cl},t+1}^\tr A_{\mathrm{cl},t+2}^\tr  (\bar{X}_{3 \to i-t+1}-X_{t+2 \to i})\notag\\
        =&{A}_{\mathrm{cl},t\to i}^\tr (\bar{X}_{i-t+1 \to i-t+1}-X_{i \to i}) \!+\! \sum_{j=t+1}^i {A}_{\mathrm{cl},t\to j-1}^\tr (\bar{K}_{j-t+1}-K_{j})^\tr B^\tr \bar{X}_{j-t+1 \to i-t+1}\notag\\
        =&{A}_{\mathrm{cl},t\to i}^\tr (\bar{P}_{i-t+2}-P_{i+1}) + \sum_{j=t+1}^i {A}_{\mathrm{cl},t\to j-1}^\tr (\bar{K}_{j-t+1}-K_{j})^\tr B^\tr \bar{X}_{j-t+1 \to i-t+1} \label{eq:X-diff}
    \end{align}
    where the last equality uses the convention that ${A}_{\mathrm{cl},t\to t}=I$.
    We also have
    \begin{align}
        &\bar{K}_{j-t+1}-K_{j}= {K}_{T-W+j-t+1}-K_{j}\notag\\
        =& (R + B^\tr P_{T-W+j-t+2} B)^{-1} B^\tr P_{T-W+j-t+2} A - (R + B^\tr P_{j+1} B)^{-1} B^\tr P_{j+1} A\notag\\
        =& (R + B^\tr P_{T-W+j-t+2} B)^{-1} B^\tr (P_{T-W+j-t+2}-P_{j+1}) A  \notag\\
        &+ \left((R + B^\tr P_{T-W+j-t+2} B)^{-1}-(R + B^\tr P_{j+1} B)^{-1}\right) B^\tr P_{j+1} A\notag\\
        =& (R + B^\tr P_{T-W+j-t+2} B)^{-1} B^\tr (P_{T-W+j-t+2}-P_{j+1}) A\notag\\
        &- (R + B^\tr P_{T-W+j-t+2} B)^{-1} B^T (P_{T-W+j-t+2}-P_{j+1}) B (R + B^\tr P_{j+1} B)^{-1} B^\tr P_{j+1} A\notag\\
        =&(R + B^\tr P_{T-W+j-t+2} B)^{-1} B^\tr (P_{T-W+j-t+2}-P_{j+1}) (A-BK_{j})\label{eq:K-diff}
    \end{align}
    where the third equality uses the identity 
    \[
    M_1^{-1}-M_2^{-1}=M_1^{-1}(M_2-M_1)M_2^{-1}
    \]
    for two invertible matrices $M_1$ and $M_2$.
    Substituting \eqref{eq:K-diff} into \eqref{eq:X-diff} then completes the proof.
\end{proof}

Based on Lemma~\ref{lem:X-diff}, Lemma~\ref{lem:STM-K-decay} and Proposition~\ref{prop:riccati-gain-convergence}, we obtain the following bound.
\begin{lem}\label{lemma:X-diff-bound}
For $t\in [T-W]$ and $t\leq i \leq {t+W-2}$, we have
\begin{align*}
    \|X_{t \to i} - \bar{X}_{1 \to i-t+1}\| &=\mathcal{O}({\gamma}_\infty^{i-t} \rho_\infty^{W-i+t}).
\end{align*}
\end{lem}

\begin{proof}
    From the expression in Lemma~\ref{lem:X-diff}, we have
    \begin{align*}
        &\|\bar{X}_{1 \to i-t+1} - X_{t \to i}\|
        \leq \|{A}_{\mathrm{cl},t\to i}^\tr \| \cdot \|\bar{P}_{i-t+2}-P_{i+1}\| \\
        + &\sum_{j=t+1}^i \! \| {A}_{\mathrm{cl},t\to j}^\tr\| \!\cdot\! \|P_{T-W+j-t+2}\!-\!P_{j+1}\| \!\cdot\! \|B (R \!+\! \!B^\tr\! P_{T-W+j-t+2} B)^{-1} B^\tr\| \!\cdot\! \|\bar{X}_{j-t+1 \to i-t+1}\|.
    \end{align*}
    First by Lemma~\ref{lem:STM-K-decay},
    \begin{equation} \label{eq:bound-bar-X}
        \|\bar{X}_{j-t+1 \to i-t+1}\| = \|\left( \bar{A}_{\mathrm{cl},i - t + 1} \cdots \bar{A}_{\mathrm{cl},j-t+2} \right)^\tr \bar{P}_{i - t + 2}\| = \mathcal{O}({\gamma}_\infty^{i-j}).
    \end{equation}
    Also by Proposition~\ref{prop:riccati-gain-convergence}, we have for $t\in [T-W]$ and $t+1\leq j < {t+W-1}$,
    \begin{align}
        \|P_{T-W+j-t+2}-P_{j+1}\|
        &\leq \|P_{T-W+j-t+2} - P_{\infty} \| + \|{P}_\infty - P_{j+1} \| \notag \\
        &= \mathcal{O}(\rho_\infty^{W-j+t-2})+\mathcal{O}(\rho_\infty^{T-j-1})\notag \\
        &= \mathcal{O}(\rho^{\min\{W-j+t-2,T-j-1\}}) \notag\\
        &=\mathcal{O}(\rho_\infty^{W-j+t}). \label{eq:bound-P-diff}
    \end{align}
    Therefore,
    \begin{align*}
        \|\bar{X}_{1 \to i-t+1} - X_{t \to i}\|&\leq  \mathcal{O}({\gamma}_\infty^{i-t}) \cdot \mathcal{O}(\rho_\infty^{W-i+t}) \\
        &\qquad \qquad + \textstyle\sum_{j=t+1}^i \mathcal{O}({\gamma}_\infty^{j-t}) \cdot \mathcal{O}(\rho_\infty^{W-j+t}) \cdot \|B\|^2 \cdot \|R^{-1}\| \cdot \mathcal{O}({\gamma}_\infty^{i-j})\\
        &= \mathcal{O}({\gamma}_\infty^{i-t} \rho_\infty^{W-i+t}) + \mathcal{O}\left({\gamma}_\infty^{i-t} \cdot \textstyle\sum_{j=t+1}^i \rho_\infty^{W-j+t} \right)\\
        &= \mathcal{O}({\gamma}_\infty^{i-t} \rho_\infty^{W-i+t})
    \end{align*}
    where the first inequality follows from \eqref{eq:bound-bar-X} and \eqref{eq:bound-P-diff}.
\end{proof}

\subsection{Proof of Lemma~\ref{lemma:mat-approx-err}} \label{subsec:proof-lem4}

\begin{proof}
The proof is based on Lemma~\ref{lemma:X-diff-bound}. Recall the feedforward gains 
\begin{align*}
    {K_{t\to i}} &= \Sigma_t^{-1}\! B^\tr (A_{\mathrm{cl},i} \cdots A_{\mathrm{cl},t+1})^\tr P_{i+1},\\
    \bar{K}_{1\to i-t+1} &= \bar{\Sigma}_1^{-1}B^\tr (\bar{A}_{\mathrm{cl},i-t+1}\cdots \bar{A}_{\mathrm{cl},2})^\tr \bar{P}_{i-t+2}
\end{align*}
for $t\in [T-W]$ and $t\leq i \leq {t+W-2}$. We can rewrite their difference as
\begin{align}
    K_{t \to i} - \bar{K}_{1 \to i-t+1} 
    &= \Sigma_t^{-1} B^\tr X_{t \to i} - \bar{\Sigma}_1^{-1} B^\tr \bar{X}_{1 \to i-t+1} \notag \\
    &= \Sigma_t^{-1}(B^\tr X_{t \to i} - B^\tr \bar{X}_{1 \to i-t+1}) + (\Sigma_t^{-1} - \bar{\Sigma}_1^{-1}) B^\tr \bar{X}_{1 \to i-t+1}. \label{eq:K-ff-diff}
\end{align}
Using Lemma~\ref{lemma:X-diff-bound}, the first term in \eqref{eq:K-ff-diff} is bounded as
\begin{align}
    \|\Sigma_t^{-1}(B^\tr X_{t \to i} - B^\tr \bar{X}_{1 \to i-t+1})\| 
    &\leq \|R^{-1}\| \cdot \|B\| \cdot \|X_{t \to i} - \bar{X}_{1 \to i-t+1}\| \notag \\
    &=\mathcal{O}({\gamma}_\infty^{i-t} \rho_\infty^{W-i+t}) \label{eq:first-term}
\end{align}
For the second term in \eqref{eq:K-ff-diff}, by the identity $\Sigma_t^{-1} - \bar{\Sigma}_1^{-1} = \Sigma_t^{-1} (\bar{\Sigma}_1 - \Sigma_t) \bar{\Sigma}_1^{-1}$,
\begin{align}
    \|(\Sigma_t^{-1} - \bar{\Sigma}_1^{-1}) B^\tr \bar{X}_{1 \to i-t+1}\|
    &\leq \|\Sigma_t^{-1}\| \cdot \|\bar{\Sigma}_1 - \Sigma_t\| \cdot \|\bar{\Sigma}_1^{-1}\| \cdot \|B\| \cdot \|\bar{X}_{1 \to i-t+1}\| \notag  \\
    &\leq \|B\|^3 \cdot \|R^{-1}\|^2 \cdot  \|\bar{P}_2 - P_{t+1}\| \cdot \|\bar{X}_{1 \to i-t+1}\|\notag \\
    &= \mathcal{O}(\rho_\infty^{W}\gamma_\infty^{i-t+1}), \label{eq:second-term}
\end{align}
where the last equality follows by arguments analogous to those in \eqref{eq:bound-P-diff} and \eqref{eq:bound-bar-X}.
Combining \eqref{eq:first-term} and \eqref{eq:second-term}, and applying Lemma~\ref{lemma:X-diff-bound}, we obtain
\begin{align*}
    \|K_{t \to i} - \bar{K}_{1 \to i-t+1}\|&= \mathcal{O}({\gamma}_\infty^{i-t} \rho_\infty^{W-i+t}) + \mathcal{O}(\rho_\infty^{W} {\gamma}_\infty^{i-t+1}) =\mathcal{O}({\gamma}_\infty^{i-t} \rho_\infty^{W-i+t}).
\end{align*}

Finally, one can verify that $K_t - \bar{K}_1 = (K_{t \to t} - \bar{K}_{1 \to 1}) A$ and thus 
\[
\|K_t - \bar{K}_1\| \leq \|K_{t \to t} - \bar{K}_{1 \to 1}\| \cdot \|A\| = \mathcal{O}(\rho_\infty^W).
\]
\end{proof}

\vspace{-3mm}
\subsection{Proof of Lemma~\ref{lem:bound-MPC-state}} \label{subsec:proof-lem5}

To prove Lemma~\ref{lem:bound-MPC-state}, we first introduce a lemma.
\begin{lem} \label{lem:MPC-state}
Consider the MPC policy given in Theorem~\ref{thm:opt-MPC}. For $t\in[T-W]$, we have
\[
u_t^{\textnormal{MPC}} = -\bar{K}_1 \tilde{x}_t - q_t(\mathbf{w}), \quad q_t(\mathbf{w}) := \textstyle\sum_{i=t}^{t + W - 2} \bar{K}_{1 \to i - t + 1} w_i,
\]
where $\tilde{x}_t=x_t-\psi(r_t)$ and $w_t=A\psi(r_t)-\psi(r_{t+1})$. Then,
\begin{itemize}[leftmargin=2em, itemsep=1pt]
    \item[\textnormal{(i)}] The state trajectory $\tilde{x}_t$ for $t\in[T-W]$ can be expressed as
    \[
    \tilde{x}_{t+1}(\mathbf{w}) = \textstyle\sum_{i=1}^t \bar{A}_{\mathrm{cl},1}^{t-i} w_i - \textstyle\sum_{i=1}^t \bar{A}_{\mathrm{cl},1}^{t-i} B q_i(\mathbf{w}).
    \]
    \item[\textnormal{(ii)}] 
    The feedforward term is uniformly bounded for all \( t\in[T-W] \) as 
    \[
    \|q_t(\mathbf{w})\| \leq D_q:=\|R^{-1}\| \cdot \|B\|\cdot \|P_\infty\| \cdot D_w \cdot (1-{\gamma}_\infty)^{-1}.
    \]
\end{itemize}
\end{lem}

\begin{proof}
    The first statement directly follows from recursive substitution. To bound $\|q_t(\mathbf{w})\|$, recall that 
    \begin{align*}
        \bar{K}_{1\to i-t+1} &= {K}_{T-W\to T-W+i-t+1} \\
        &= \Sigma_{T-W}^{-1} B^\tr ({A}_{\mathrm{cl},T-W+i-t+1} \cdots {A}_{\mathrm{cl},T-W+1})^\tr P_{T-W+i-t+2}
    \end{align*}
    defined for $t\leq i\leq t+W-2$ with $t\in[T-W]$. Therefore, by Lemma~\ref{lem:STM-K-decay},
    \begin{align*}
        \|q_t(\mathbf{w})\| &= \left\|\textstyle\sum_{i=t}^{t + W - 2} \bar{K}_{1 \to i - t + 1} w_i\right\|\\
        &\leq \|\Sigma_{T-W}^{-1}\| \cdot \|B\|\cdot \max_{i}\|P_{T-W+i-t+2}\| \cdot D_w \cdot \textstyle\sum_{i=t}^{t+W-2}\|{A}_{\mathrm{cl},T-W\to T-W+i-t+1}\|\\
        &\leq \|R^{-1}\| \cdot \|B\|\cdot \|P_\infty\| \cdot D_w \cdot \textstyle\sum_{i=t}^{t+W-2} \mathcal{O}({\gamma}_\infty^{i-t+1}) \leq D_q.
    \end{align*}
\end{proof}

\begin{proof}[Lemma~\ref{lem:bound-MPC-state}]    
    We bound the norm of the lifted MPC trajectory as follows:
    \begin{align*}
        \|\tilde{x}_{t+1}(\mathbf{w})\| &\overset{(a)}{=} \left\|\textstyle\sum_{i=1}^t \bar{A}_{\mathrm{cl},1}^{t-i} w_i - \sum_{i=1}^t \bar{A}_{\mathrm{cl},1}^{t-i} B q_i(\mathbf{w})\right\|\\
        &\overset{(b)}{\leq} \kappa_\infty \left(1+\|A\|+\max_{i}\|q_i(\mathbf{w})\|\right)\cdot  \textstyle\sum_{i=1}^t {\gamma}_\infty^{t-i}\\
        &\overset{(c)}{\leq} \kappa_\infty (1+\|A\|+D_q) \cdot (1-{\gamma}_\infty)^{-1}:=D_x,
    \end{align*}
    where $(a)$ follows from Lemma~\ref{lem:MPC-state}, $(b)$ applies Lemma~\ref{lem:MPC-stability}, which ensures that for any $W \geq \Delta_{\mathrm{stab}}$ the decay bound $\|\bar{A}_{\mathrm{cl},1}^{t-i}\|\leq\kappa_\infty \gamma_\infty^{t-i}$ holds, and $(c)$ uses again Lemma~\ref{lem:MPC-state}.
\end{proof}

\subsection{Proof of Proposition~\ref{prop:gain-deviation}}\label{app:proof-improved-bound}
\begin{proof}
We first bound the feedback gain deviation. From Lemma~\ref{lem:bound-MPC-state}, we have
    \begin{align*}
    \textstyle\sum_{t=1}^{T-W} \left\|(K_t - \bar{K}_1)(x_t - \psi(r_t))\right\|_{\Sigma_t}^2 
    &\leq D_x^2 \cdot \max_t \|\Sigma_t\|  \cdot \textstyle\sum_{t=1}^{T-W} \|K_t - \bar{K}_1\|^2 = \mathcal{O}(\rho_\infty^{2W} T),
    \end{align*}
    where we use the convergence bound \( \|K_t - \bar{K}_1\| = \mathcal{O}(\rho_\infty^W) \) from Proposition~\ref{prop:riccati-gain-convergence}.

We next bound the feedforward gain deviation. Since \( \|w_i\| \leq D_w \), we get
\begin{align*}
    &\textstyle\sum_{t=1}^T \left\|\textstyle\sum_{i=t}^{t+W-2} (K_{t\to i} - \bar{K}_{1\to i-t+1}) w_i\right\|_{\Sigma_t}^2 \\
    \leq& D_w^2 \cdot \max_t \|\Sigma_t\| \cdot \textstyle\sum_{t=1}^T \left( \sum_{i=t}^{t+W-2} \|K_{t\to i} - \bar{K}_{1\to i - t + 1}\| \right)^2 \\
    \leq& D_w^2 \cdot \max_t \|\Sigma_t\| \cdot \textstyle\sum_{t=1}^T \left( \textstyle\sum_{i=t}^{t+W-2} \mathcal{O}({\gamma}_\infty^{i-t} \rho_\infty^{W-i+t}) \right)^2 \\
    \leq& \mathcal{O}(1) \cdot \textstyle\sum_{t=1}^T \left( \textstyle\sum_{j=0}^{W-2} \mathcal{O}(\gamma_\infty^j \rho_\infty^{W - j}) \right)^2 \\
    \leq& \mathcal{O}(1) \cdot \textstyle\sum_{t=1}^T \left( \sum_{j=0}^{W-2} \mathcal{O}(\lambda_\infty^W) \right)^2 = \mathcal{O}(W^2 \lambda_\infty^{2W} T),
\end{align*}
where $\lambda_\infty := \max\{\gamma_\infty, \rho_\infty\}$, and we use the bound from Lemma~\ref{lemma:mat-approx-err} in the second inequality.
Note that when $\gamma_\infty < \rho_\infty$, we have
$
\textstyle\sum_{j=0}^{W-2} \mathcal{O}\left( ( \gamma_\infty/{\rho_\infty})^j \right) = \mathcal{O}\left((1-{\gamma_\infty}/{\rho_\infty})^{-1}\right) = \mathcal{O}(1)
$
and hence the bound can be improved to $\mathcal{O}(\rho_\infty^{2W}T)$.
\end{proof}

\vspace{-3mm}

\section{Additional Experiments: Two-wheeled Robots}
\label{app:exp-details}
We here illustrate that the proposed online controller \eqref{eq:DDPC} can be adapted to work for nonlinear systems that are not Koopman-linearizable. We consider a two-wheeled mobile robot with nonlinear kinematic dynamics \citep{li2019online}:
\vspace{-1mm}
\begin{align*}
    z_{x, t+1}  &= z_{x, t} + \Delta t \cdot \cos(z_{\delta, t}) \cdot v_{t}, \\ 
    z_{y, t+1} &= z_{y, t} + \Delta t \cdot \sin(z_{\delta, t}) \cdot v_{t}, \\
    z_{\delta, t+1}  &= z_{\delta, t} + \Delta t \cdot w_t,
\end{align*} 
where $(z_x, z_y)$ denotes the robot's position, $v$ and $w$ represent tangential and angular velocities, $z_\delta$ is the heading angle (relative to the $x$-axis), and $\Delta t = 0.025\, \mathrm{s}$.
The objective is to track a heart-shaped reference trajectory $(\mathbf{r}_x, \mathbf{r}_y, \mathbf{r}_\delta)$ given below while minimizing control effort:
\[
\begin{aligned}
r_{x,t} & = 16 \sin^3(t-6), \\
r_{y,t} & = 13\cos(t) - 5\cos(2t-12)-2\cos(3t-18)-\cos(4t-24), \\
r_{\delta, t} & = \mathrm{arctan}\left(\frac{r_{y, t+1}-r_{y, t}}{r_{x, t+1}-r_{x, t}}\right).
\end{aligned}
\]
We set the cost matrices
$
\tiny
Q_z = \begin{bmatrix}
    1 & 0 & 0 \\ 0 & 1 & 0 \\ 0 & 0 & 2
\end{bmatrix}$
and 
$\tiny R = 1.3\times 10^{-3} I_2.$
As this nonlinear model is generally not Koopman-linearizable, we introduce a regularization term for implicit model selection and a slack variable for numerical stability \citep{dorfler2022bridging}. The regularized DDPC then becomes
\begin{equation*} \label{eq:reg-DDPC}
\begin{aligned}
     \min_{u_{1:W|t}, z_{1:W|t}, g, \sigma_z} & \quad \textstyle\sum_{i=1}^{W}  \|z_{i|t} - r_{t+i-1}\|_{Q_z}^2 + \|u_{i|t}\|_R^2 + \lambda_g \|g\|_1 + \lambda_z \|\sigma_z\|_2^2  \\
     \mathrm{s.t.} & \quad H_\D g = [
        \mathbf{u}_{\ini}^\tr, u_{1:W|t}^\tr, (\mathbf{z}_{\ini}+\sigma_z)^\tr, z_{1:W|t}^\tr
        ]^\tr.
     \end{aligned} \tag{Reg-DDPC}
\end{equation*}
When the data library $H_\D$ is collected from a nonlinear system that is not Koopman-linearizable, it typically has full row rank, rendering the equality constraint ineffective.
To address this, we employ the $l_1$-norm regularizer $\|g\|_1$ to encourage the optimization to select more ``important" trajectories for prediction.
This serves as a convex relaxation of the $l_0$-norm to reduce model complexity.

In addition, modeling error may cause the initial trajectory to fall outside the range space of $H_\D$.
To maintain feasibility, we introduce a slack variable $\sigma_z$ and penalize large deviations from the initial trajectory through a relatively large weight $\lambda_z$. In our simulation, we set $\lambda_g = 2$ and $\lambda_z = 3\times 10^6$.
Also note that for general nonlinear systems, the collected data captures only the local dynamics over the data collection region.
Therefore, the data-driven constraint in \eqref{eq:reg-DDPC} can be viewed as a local linear approximation.
Thus, to improve prediction accuracy and tracking performance, our strategy is to switch the data library based on the robot's current state.

\begin{figure}[t]
    \centering
    \includegraphics[width=0.31\textwidth]{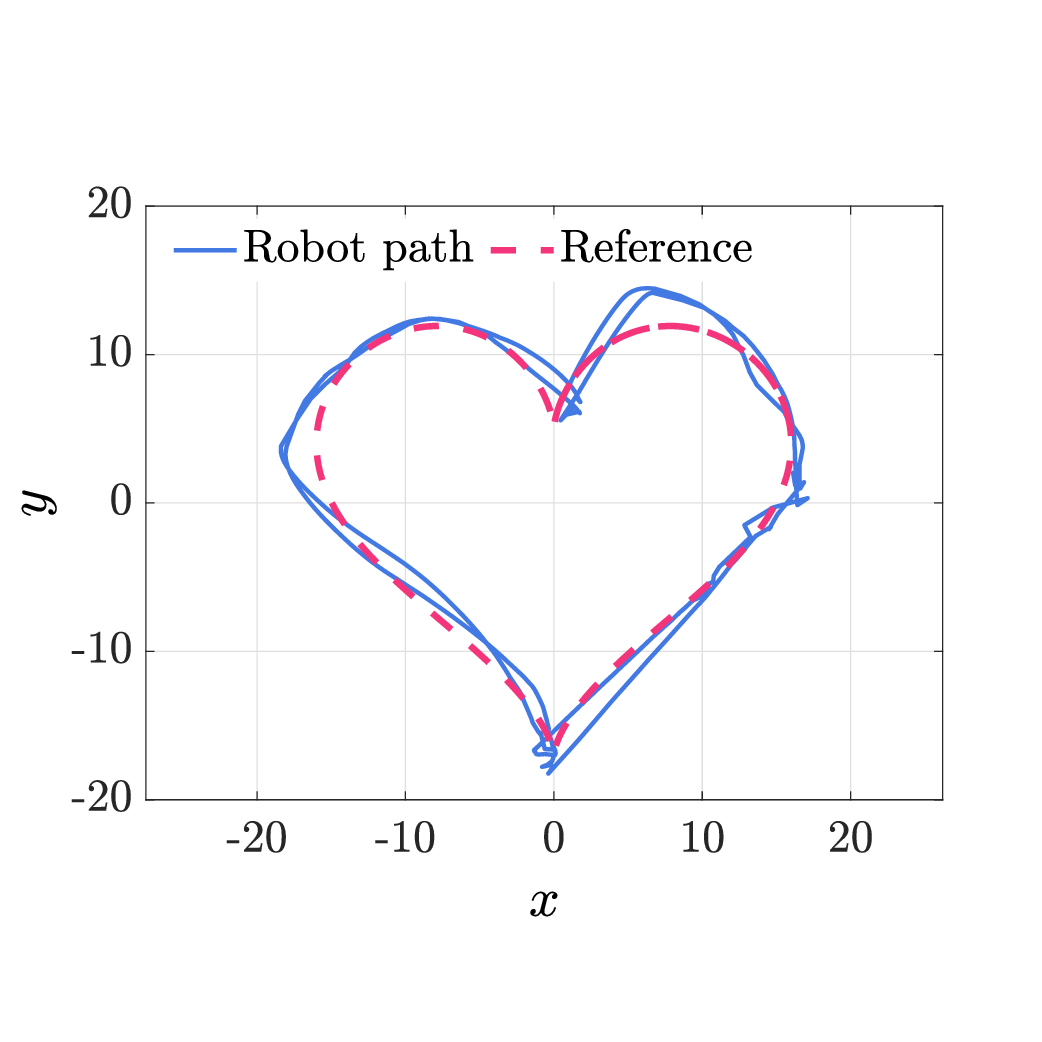} \label{subfig:Nl_W=6} \hspace{1mm}
    \includegraphics[width=0.31\textwidth]{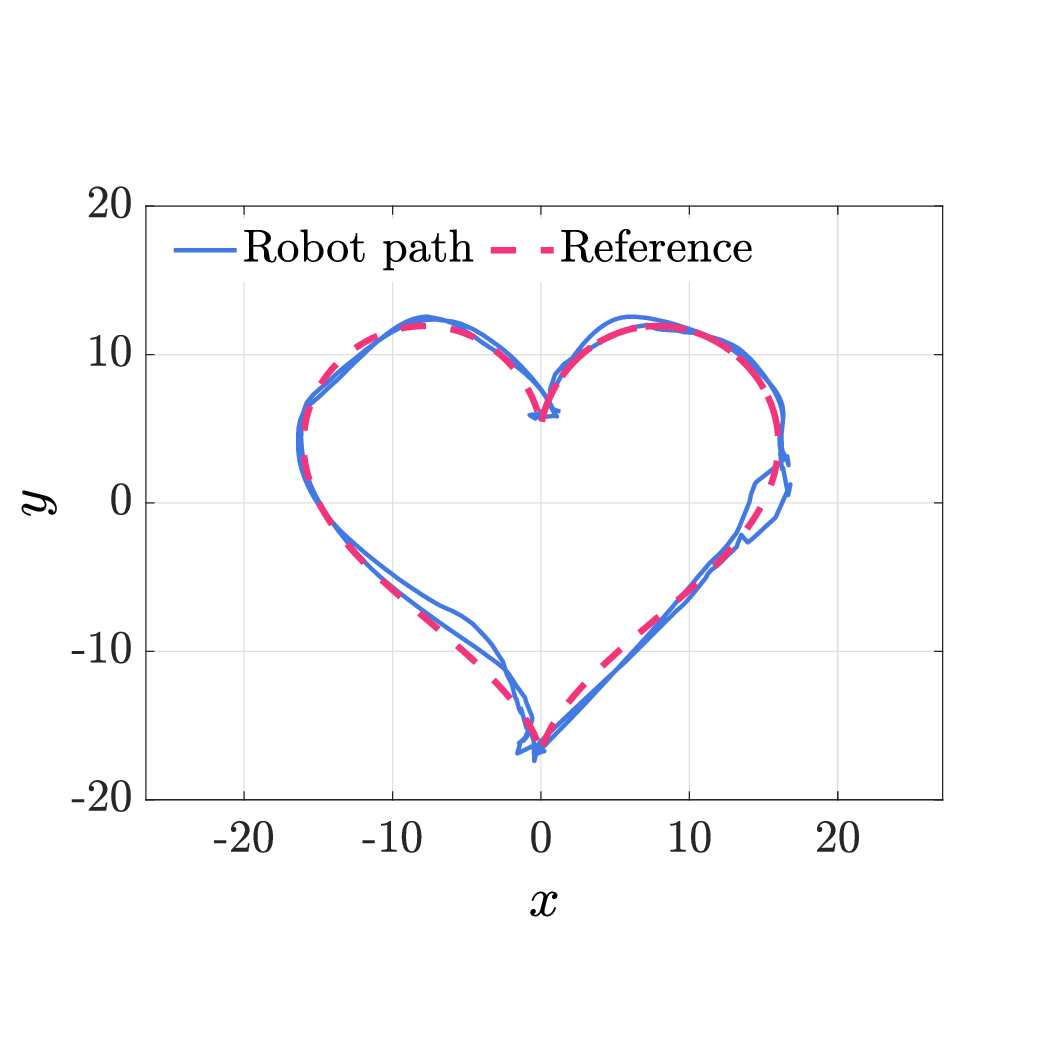} \label{subfig:Nl_W=9} \hspace{1mm}
    \includegraphics[width=0.31\textwidth]{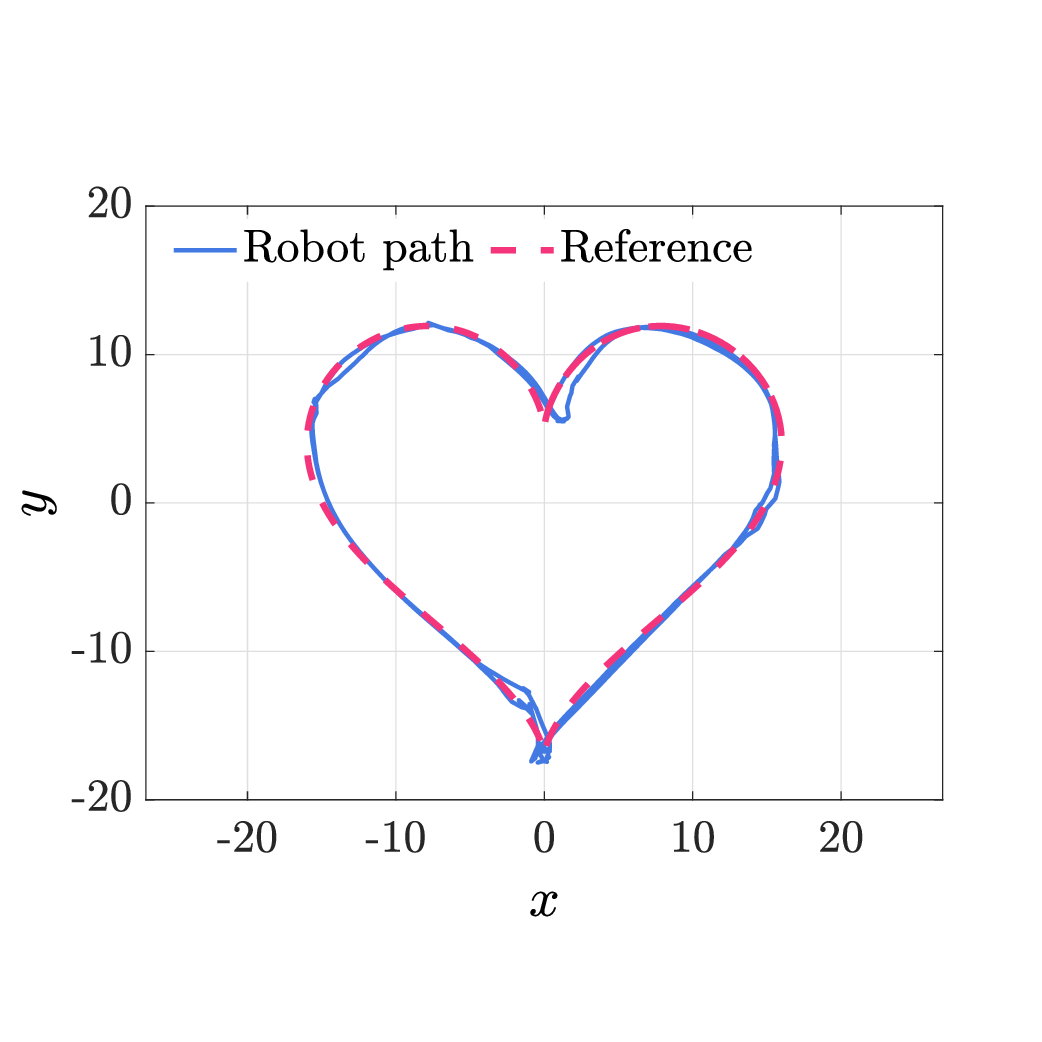} \label{subfig:Nl_W=12}
    \vspace{-2mm}
    \caption{Tracking performance of the two-wheel robot under varying prediction horizons: $W=6$ (left), $W=9$ (middle), and $W=12$ (right). The red curve denotes the reference trajectory, and the blue curve shows the robot's actual path.}
    \vspace{-2mm}
    \label{fig:tracking_nonlinear}
\end{figure}

Since the nonlinearity of the two-wheeled robot mainly arises from the coupled terms $\sin(z_{\delta}) \cdot v$ and $\cos(z_{\delta}) \cdot v$ in the dynamics of $z_x$ and $z_y$, the switching strategy is based on the robot's orientation.
Specifically, during offline data collection, we collect $4$ trajectories of length $1500$ with initial states $(0, 0, \frac{\pi}{4}), (0, 0, \frac{3\pi}{4}), (0, 0, \frac{5\pi}{4})$ and $(0, 0, \frac{7\pi}{4})$.
The input signals $v$ and $w$ are uniformly sampled from $[10, 20]$ and $[-\frac{\pi}{6}, \frac{\pi}{6}]$, respectively.
For each prediction horizon $W$, we construct $4$ corresponding data libraries $H_\D$ with $T_\ini = 5$, each approximately representing the local behavior of the nonlinear system within orientation intervals $[ 0, \frac{\pi}{2}),[ \frac{\pi}{2}, \pi), [\pi, \frac{3\pi}{2})$, and $[ \frac{3\pi}{2}, 2\pi)$. During online tracking, the data library associated with the robot’s current orientation is selected to improve tracking performance.

The results of tracking the heart-shaped curve over two cycles with prediction horizons $W=6$, $9$, $12$ are shown in Figure~\ref{fig:tracking_nonlinear}.
We observe that the robot tracks the reference trajectory relatively well in smooth segments but exhibits noticeable deviation or oscillation at sharp turns. As we can see, a longer prediction horizon enhances tracking accuracy in smooth regions by reducing drift and allows the robot to recover more quickly after deviating at sharp corners.

\end{document}